\documentclass[preprint,12pt]{elsarticle}

\usepackage{amssymb,amsmath}
\usepackage{latexsym}   
\usepackage{amsthm}     
\usepackage{xcolor}
\usepackage{graphicx}
\usepackage{booktabs}
\usepackage[normalem]{ulem}
\usepackage[ruled,vlined]{algorithm2e}

\usepackage{tikz,pgfplots,tikz-network}
\usepackage{pgfplotstable}
\pgfplotsset{compat = newest}
\usetikzlibrary{graphs,graphs.standard}
\usetikzlibrary{arrows.meta}

\newtheorem{theorem}{Theorem}[section]
\newtheorem{remark}[theorem]{Remark}

\newtheorem{myexample}{Example}
\newtheorem{corollary}[theorem]{Corollary}
\newtheorem{lemma}[theorem]{Lemma}
\newtheorem{proposition}[theorem]{Proposition}

\def\R{{\mathbb R}}

\def\Rnn{{\mathbb R}^{n\times n}}

\def\Rn0{{\mathbb R}^n_0}
\def\uno{\mathbf 1}
\def\fro{\mathrm{F}}
\def\che{{\max}}
\newcommand\wh{\widehat}
\newcommand{\hinv}[1]{{#1}^{[-1]}}

\DeclareMathOperator{\cau}{\tt Cauchy}
\renewcommand{\vec}{\mathrm{vec}}
\def\D{{\mathcal D}}

\journal{Linear Algebra Appl.}

\begin{document}

\begin{frontmatter}

\title{Advances on the recovery of (perturbed) Cauchy matrices}

\author[1]{P.~Boito}
\ead{paola.boito@unipi.it}
\author[2]{D.~Fasino\corref{df}}
\ead{dario.fasino@uniud.it}
\author[1]{B.~Meini}
\ead{beatrice.meini@unipi.it}
\affiliation[1]{organization={Dipartimento di Matematica,
Universit\`a di Pisa},
addressline={Largo Bruno Pontecorvo, 5},
city={Pisa},
citysep={},
postcode={56127},
country={Italy}}
\affiliation[2]{organization={Dipartimento di Scienze Matematiche, Informatiche e Fisiche,
Universit\`a degli Studi di Udine},
addressline={Via delle Scienze, 206},
city={Udine},
citysep={}, 
postcode={33100},
country={Italy}}

\cortext[df]{Corresponding author}

\begin{abstract}
    Given a (possibly approximate) Cauchy matrix, how can we efficiently compute its generators? Expanding on previous work by Liesen and Luce [Linear Algebra and its Applications, 2016], we present a general family of algorithms for Cauchy parameter recovery, together with new error estimates. We also introduce a displacement-based approximation, which leads to a new algorithm for Cauchy parameter recovery. Numerical experiments show that the algorithm based on the displacement approximation is generally more accurate than the other algorithms.
\end{abstract}


\end{frontmatter}

\section{Introduction}

Cauchy matrices are a family of structured matrices appearing 
in interpolation and approximation problems with rational functions. Moreover, they belong to the class of displacement structured matrices and, as such, they share notable computational properties with other 
structured matrix families, e.g., Toeplitz, Vandermonde, and Hankel matrices \cite{kailath}. In fact, several algebraic computations involving Cauchy matrices can be performed efficiently using fast algorithms \cite{kailath_book}. This is because, although a Cauchy matrix is fully populated by nonzero entries, it can be exactly described by a set of parameters growing linearly with the dimension. Cauchy matrices have been characterized in \cite{MR3314339}
as transition matrices between the eigenbases of two diagonalizable matrices that differ by a rank-one term. Moreover, row- and column-scaled versions 
of Cauchy matrices also appear in the numerical solution of secular equations and as eigenvector matrices of particular symmetric quasiseparable matrices \cite{MR4530350}.

The present work is motivated by the results of Liesen and Luce \cite{liesen2016fast}, who addressed the problem of determining whether a given matrix is Cauchy or can be approximated by a matrix with that structure.
To this goal, the authors of \cite{liesen2016fast} presented two algorithms. The first determines the parameters that define a Cauchy matrix using only the values in its first row and column.
The second computes the parameters 
of a Cauchy matrix that best approximates a perturbed data matrix, in some sense.
Both algorithms have optimal complexity, and the second is accompanied by {\em a posteriori} error bounds, i.e., upper bounds on the approximation errors based on the evaluation of suitably defined residuals. These algorithms are based on a characteristic property of Cauchy matrices: the matrix obtained by entrywise reciprocation is rank-two and has a specific structure. Therefore, the problems discussed by Liesen and Luce are somewhat related to the issue of representing rank-two matrices appropriately. In the case of generic matrices with non-negative elements, this latter problem is addressed in great depth in \cite{R2NMF}.

The problem of reconstructing a matrix having some kind of structure from noisy data is a well-known topic in numerical linear algebra that has been developed recently along different directions, because of the applications in mathematical modeling, signal processing, data compression, time series analysis, etc.
For example, in \cite{MR3598578}
the authors
address the problem of best approximation of a given matrix by a matrix of lower rank in the elementwise maximum norm.
Comprehensive summaries on the construction of a structured low-rank matrix that is nearest to a given matrix are given in
\cite{MR1987719, MR2530933}.
Also the reconstruction of perturbed matrices with displacement structures, notably Toeplitz, Hankel, and Vandermonde matrices, has received some attention, see e.g., \cite{MR4255495, Halikias2024, WANG2016133}. The computational approaches are usually based on the minimization of a possibly non-convex 
objective functional
which measures the error between the data matrix and the approximating one, which may be subject to both rank and structural constraints.
Analogous techniques are also employed for solving completion problems with displacement-structured matrices, see \cite{MR4125767}. 
A broad family of algorithms for recovering linearly parameterized matrices  via matrix-vector products is analyzed in \cite{amsel26} and \cite{Halikias2024}
to characterize the computational complexity of the recovery problem.

In this work, we take a deeper look at the problem of recovering a Cauchy matrix from perturbed data.
After collecting some preliminary notions and results on Cauchy matrices in the next section, Section \ref{sec:generalfamily} proposes a unified framework for the description and analysis of the algorithms in \cite{liesen2016fast}, based on a suitable family of projectors onto a subspace of rank-2 matrices. 
This framework allows us to introduce a parametrized algorithm, here called Algorithm \ref{alg:generaluv}, for 
solving our parameter recovery problem that includes both algorithms discussed by Liesen and Luce as particular cases. Moreover, we provide error bounds on the recovered parameters that are  {\em a priori}, that is, intrinsic to the algorithm and do not depend on residual norms. For notational simplicity, we limit ourselves to considering real square matrices. Extending our results to rectangular and complex matrices only requires straightforward notational amendments and minor adaptations. On the other hand, we consider measuring approximation errors in both the Frobenius norm and the elementwise maximum norm.

Section \ref{sec:CUR} provides a further analysis of Algorithm 1 in \cite{liesen2016fast} that relies on CUR factorization theory and provides new error estimates involving the singular values of an augmented matrix.
In Section \ref{sec:displacementbased} we formulate a measure of `Cauchyness' based on the displacement characterization of the Cauchy structure, and propose a recovery strategy based on the minimization of that measure, which leads to Algorithm \ref{alg:last}. 
Finally, in Section~\ref{sec:numerical}, we present the results of a series of numerical experiments to illustrate the performance of the various algorithms.
One such example illustrates a potential application involving the reconstruction of perturbed Cauchy matrices.
As a side note, we discuss in the Appendix a parametrization of (entrywise reciprocated) Cauchy matrices that is perfectly well-conditioned.
Compared to the one originally presented in 
\cite{liesen2016fast}, the new parametrization 
has a tighter error bound when recovering a perturbed Cauchy
matrix via Algorithm 2.

\subsection{Notation}

The following notation will be used throughout this paper.
The maximum (Chebyshev) norm and the Frobenius norm of a matrix $A\in\Rnn$ 
are defined as
$$
   \|A\|_\che = \max_{i,j=1,\ldots,n}|A_{ij}| , \qquad
   \|A\|_\fro = \bigg(\sum_{i,j=1}^n A_{ij}^2\bigg)^{1/2} ,
$$
respectively.
These are the matrix norms most commonly encountered in matrix approximation problems, see e.g., \cite{MR2443975,MR3598578,GoTy01,MR2530933}.
They have the useful property that the matrix norm of $A$ is equal to some $p$-norm of $\vec(A)$.
We sometimes use the symbol $\|\cdot\|_\star$ 
to denote any of the above matrix norms.
These norms share the property presented in the following lemma.

\begin{lemma}   \label{lem:easy}
If $A,B\in\Rnn$ are two matrices such that
$|A_{ij} - B_{ij}|/|A_{ij}| \leq \alpha$ for every $i,j=1,\ldots,n$
for some $\alpha \geq 0$,
then $\|A - B\|_\star/\|A\|_\star \leq \alpha$.
\end{lemma}

An $n\times n$ identity matrix is denoted $I_n$ or simply $I$ if the size is evident.
The symbol $\uno$ denotes the all-ones vector 
of appropriate size, $\uno=(1,\ldots,1)^T\in\R^n$. We also write $\uno_n$
to explicitly indicate the size.
Denote by $\vec$ the vectorization operator such that $\mathrm{vec}(A)$ is the vector in $\R^{n^2}$ obtained by stacking the columns of $A$.
Note that $\|A\|_\che = \|\mathrm{vec}(A)\|_\infty$ and 
$\|A\|_\fro = \|\mathrm{vec}(A)\|_2$.
The symbol $\otimes$ denotes the Kronecker product. The following well-known identity establishes a relationship between
matrix multiplication, Kronecker product, and vectorization:
\begin{equation}   \label{eq:ABC}
   \mathrm{vec}(ABC) = (C^T\otimes A)\mathrm{vec}(B) .
\end{equation}
Let $M$ and $N$ be two matrices with the same number of columns. The MATLAB-style notation $[M;N]$ denotes the matrix obtained by stacking $M$ on top of $N$.
Conversely, if $M$ and $N$ have the same number of rows, then $[M\ N]$ indicates the matrix obtained by concatenating the rows of $M$ and $N$.
For a matrix $X\in\Rnn$ with no zero entries we denote as $\hinv{X}\in\Rnn$ the entrywise reciprocal matrix,
$$
   \hinv{X}_{ij} = 1/X_{ij} .
$$
For any vector $x\in\R^n$ we denote $\mathrm{Diag}(x)$ 
the $n\times n$ diagonal matrix with the entries of $x$ in the main diagonal. 

\section{Preliminaries on Cauchy matrices and the recovery of Cauchy points}\label{sec:preliminaries}
A matrix $C \in\Rnn$ is a {\em Cauchy matrix} if
$$
   C_{ij} = \frac{1}{x_i-y_j} 
$$
for real numbers $x_1,\ldots,x_n$ and $y_1,\ldots,y_n$
such that $x_i \neq y_j$ for $i,j=1,\ldots,n$.
We say that the vectors $x = [x_i]_{i=1,\ldots,n}$ and 
$y =  [y_i]_{i=1,\ldots,n}$
are {\em Cauchy points} of the matrix $C$, 
and we adopt the notation $C = \cau(x,y)$.
Note that the Cauchy points are defined up to an additive constant.
In fact, it is not hard to check that $\cau(x,y) = \cau(x+\alpha\uno,y+\alpha\uno)$ for every $\alpha\in\R$. 
Following \cite{liesen2016fast} we say that the Cauchy points $x$ and $y$
of a Cauchy matrix $C$
are {\em normalized} if $\sum_{i=1}^n (x_i^2 + y_i^2)$ is minimal among all possible Cauchy points.
It can be easily verified that, given $C = \cau(x,y)$,
the vectors $\hat x = x -\alpha\uno$ and $\hat y = y -\alpha\uno$
with $\alpha = \sum_i(x_i+y_i)/(2n)$ are the normalized Cauchy points of $C$.

Let $\mathcal{D}\subset \Rnn$ be the set
\begin{equation}   \label{eq:defD}
   \mathcal{D} = \{ x\uno^T - \uno y^T\in\Rnn \, |\, x,~y\in\R^n\} ,
\end{equation}
which is a vector subspace of $\Rnn$ of dimension $2n-1$.
Based on this definition, we can say that a matrix $A \in\Rnn$ is a Cauchy matrix if and only if it has no zero entries
and $\hinv{A}\in\mathcal{D}$.
Let us also introduce the matrix-valued operator
$H:\R^n\times \R^n\to\Rnn$ defined as
$$
   H(x,y) = x\uno^T - \uno y^T .
$$
It holds $\mathrm{Range}(H) = \mathcal{D}$.
Moreover, if $x$ and $y$ are Cauchy points then 
$\cau(x,y) = \hinv{H(x,y)}$.
However, not all matrices in $\mathcal{D}$ are entrywise reciprocals of Cauchy matrices.
We call {\em generators} of a matrix $D\in\mathcal{D}$
any vector pair $(x,y)$ such that $D = H(x,y)$. Furthermore, we say that $(x,y)$ are {\em normalized} if $\sum_{i=1}^n (x_i^2 + y_i^2)$ is minimal.

The following Algorithm \ref{alg:LL1}, borrowed from \cite{liesen2016fast},
recovers the normalized Cauchy points of a given Cauchy matrix.
The computed vectors $x$ and $y$ are 
identified by equating the entries in the first row and column of 
$H(x,y)$ and the entrywise inverse of the input matrix.
If the input matrix $A$ is Cauchy 
then $x$ and $y$ are normalized Cauchy points such that 
$A = \cau(x,y)$. 
However, Algorithm \ref{alg:LL1}
can also be applied to a generic matrix with nonzero entries, in which case 
the vectors $x$ and $y$ are normalized generators of a matrix in $\mathcal{D}$.

\begin{algorithm}[ht]   \label{alg:LL1}
\LinesNumbered 
\SetNlSty{texttt}{}{.} 
\DontPrintSemicolon
\caption{Recovery of normalized generators}   
\KwIn{Matrix $A = (A_{ij})$ with no zero entries}
\KwOut{Normalized generators $x,y$ }
\BlankLine
\For{$i = 1$ \KwTo $n$}
   {
   $y_i = -1/A_{1i}$ \;  
   }
$x_1 = 0$\;
\For{$i = 2$ \KwTo $n$}{
$x_{i} = 1/A_{i1} + y_1$ \;
}
$\alpha = (\sum_{i=1}^n x_i + y_i)/(2n)$\;
$x = x - \alpha\uno$\;
$y = y - \alpha\uno$\;
\end{algorithm}

Now, suppose that $A$ is a perturbed Cauchy matrix, that is, 
$A = B + E$ where $B$ is Cauchy and $E$ is a perturbation matrix
with `small' entries. Aiming at recovering the matrix $B$,
the authors of \cite{liesen2016fast}
propose the Algorithm \ref{alg:LL2} below 
which computes vectors $x$ and $y$ such that 
$H(x,y)$ is the solution of 
\begin{equation}   \label{eq:minF}
   \min_{X\in\mathcal{D}}\|\hinv{A} - X\|_\fro^2 .
\end{equation}

Algorithm \ref{alg:LL2} is obtained by converting \eqref{eq:minF} to a standard least squares problem via vectorization, and deriving an explicit formula for the 
least-norm solution.
The correctness of this algorithm is shown in 
\cite[Thm.\ 3.2]{liesen2016fast}, which also provides a necessary and sufficient condition for the inequality $x_i\neq y_j$ to hold for $i,j=1,\ldots,n$. In this case, $\cau(x,y)$ can be considered as an approximation of the hidden Cauchy matrix $B$.
On the other hand, it is also shown in \cite{liesen2016fast}
that these algorithms
may fail to provide Cauchy points
when the data matrix is noisy, since the condition $x_i\neq y_j$ may not be fulfilled for all $i,j=1,\ldots,n$ in some cases. The following theorem shows that,
if $\hinv{A}$ is quite close to $\mathcal{D}$
then $A$ can be approximated by a Cauchy matrix,
and also provides a relative normwise estimate 
of the approximation error,
see \cite[Thm.\ 3.5]{liesen2016fast}.

\begin{algorithm}[ht]   \label{alg:LL2}
\LinesNumbered 
\SetNlSty{texttt}{}{.} 
\DontPrintSemicolon
\caption{Normalized generators from the solution of \eqref{eq:minF}} 
\KwIn{Matrix $A = (A_{ij})$ with no zero entries}
\KwOut{Normalized generators $x,y$ }
\BlankLine
$r = \hinv{A}\uno/n$\;
$c = {\hinv{A}}^T\uno/n$\;
$\alpha = (\sum_{i=1}^n r_i)/(2n)$\;
$x = r - \alpha\uno$\;
$y = \alpha\uno - c$\;
\end{algorithm}

\begin{theorem}   \label{thm:LL1}
Let $A\in\R^{n\times n}$ be a matrix with no zero entries and let 
$D \in\mathcal{D}$ 
be a matrix such that 
\begin{equation}   \label{eq:hyp}
   | 1 - A_{ij}D_{ij}| \leq \beta < 1 
\end{equation}
for $i,j = 1,\ldots,n$.
Then $D =H(x,y)$ with $\min_{i,j}|x_i - y_j| \geq (1-\beta)/\|A\|_\che$. Moreover, if $C = \cau(x,y)$ then
$$
   \frac{\|A - C\|_\star}{\|A\|_\star} \leq \frac{\beta}{1-\beta} .
$$
\end{theorem}

This theorem has been proven in \cite{liesen2016fast} in the $\star = \fro$ case.
However, the original proof incidentally shows that
$$
   \frac{|A_{ij} - C_{ij}|}{|A_{ij}|} \leq \frac{\beta}{1-\beta} ,
$$
which is easily deduced from \eqref{eq:hyp}.
By this inequality and Lemma \ref{lem:easy}, we can conclude that Theorem \ref{thm:LL1} is also true in the Chebyshev norm.  
The first part of the claim  
shows that the vectors $x$ and $y$ 
are Cauchy points, whereas the last part
provides a relative normwise approximation error of $C$ with respect to $A$. As $C_{ij} = 1/D_{ij}$, the leftmost term in \eqref{eq:hyp} can be rewritten as
\begin{equation}   \label{eq:rearranged}
   | 1 - A_{ij}D_{ij}| = |A_{ij}| \bigg| \frac{1}{A_{ij}}
   - \frac{1}{C_{ij}} \bigg| 
   = \frac{|C_{ij} - A_{ij}|}{|C_{ij}|}.
\end{equation}
Thus the constant $\beta$ in \eqref{eq:hyp} is a bound
on the entrywise relative error between the matrices $\hinv{A}$ and $\hinv{C}$ or, equivalently, 
between $C$ and $A$.
In passing, we note that Theorem \ref{thm:LL1} holds
for every matrix $D\in\mathcal{D}$ that fulfils the hypothesis \eqref{eq:hyp}, not just the matrix attaining the minimum in \eqref{eq:minF}.
Furthermore, it is not difficult to complement Theorem \ref{thm:LL1} with the following result, which provides a sort of stability estimate for the recovery of the matrix $C$.

\begin{corollary}
In the same hypotheses and notations of Theorem \ref{thm:LL1}, it holds
$\|C\|_\star\leq \|A\|_\star/(1-\beta)$, 
where $\|\cdot\|_\star$ denotes either the Frobenius or the Chebyshev norm.
\end{corollary}

\begin{proof}
    Using \eqref{eq:hyp} and \eqref{eq:rearranged},
    for all $i,j=1,\ldots,n$
    we have 
$$
   \beta |C_{ij}| \geq |C_{ij} - A_{ij}|
   \geq |C_{ij}| - |A_{ij}| .
$$
Thus $|A_{ij}| \geq |C_{ij}| (1-\beta)$,
and the claim follows.
\end{proof}

\section{A general family of algorithms for Cauchy parameter recovery}\label{sec:generalfamily}

In this section, we uncover a common structure of the two algorithms in 
the preceding section. This structure allows us to devise a parametrized
algorithm for the approximation of a perturbed Cauchy matrix,
which includes Algorithm \ref{alg:LL1} and \ref{alg:LL2}
as particular cases.
Furthermore, we provide {\em a priori} error bounds on the 
approximation computed by this algorithm.
A close look at Algorithms \ref{alg:LL1} and \ref{alg:LL2} reveals that they implement linear projectors onto $\mathcal{D}$. 
To reveal the common structure of these projectors, we introduce the 
matrix operator $\Delta:\R^{n\times n}\to \R^{n\times n}$ given by 
$\Delta(X) = X - MXN^T$ for some auxiliary matrices $M$ and $N$.
In the literature of structured matrices, this operator is 
known as a Stein displacement operator; see \cite{kailath_book,MR533501,kailath} and also \cite{MR3681370}.

\begin{theorem}   \label{thm:vw}
The displacement operator $\Delta(X) = X - MXN^T$ is a projector onto $\D\subset\Rnn$
if and only if
there exist $v,w\in\R^n$ such that 
$\uno^Tv = \uno^Tw = 1$, $M = I - \uno v^T$
and $N = I - \uno w^T$.
\end{theorem}

\begin{proof} 
Let $A,B\in\Rnn$ be such that $M=I-A$, $N=I-B$. Suppose first that
for any $X\in\mathbb{R}^{n\times n}$
there exist $x,y\in\R^n$ such that 
\begin{equation}\label{eq:block}
    X - MXN^T
    =x\uno^T - \uno y^T .
\end{equation}
The left-hand side of \eqref{eq:block}
can be written as $AX + (I-A)XB^T$.
Let $i,j\in\{1,\ldots,n\}$ with $i\neq j$ be fixed,
and consider the matrix $X = ze_i^T$ where $z\in\R^n$ is arbitrary.
Multiplying both sides of \eqref{eq:block} by $e_i-e_j$ and simplifying, we obtain
on the left-hand side
\begin{align*}
   AX(e_i-e_j) + (I-A)XB^T(e_i-e_j) 
   & = Az + (I-A)ze_i^T B^T(e_i-e_j) \\
   & = Az + \xi_{ij}(I-A)z \\
   & = [(1-\xi_{ij})A + \xi_{ij}I]z ,
\end{align*}
with $\xi_{ij} = B_{ii} - B_{ji}$ and 
$$
   (x\uno^T - \uno y^T)(e_i-e_j) = \uno (y_j-y_i) 
$$
on the right-hand side. 
Therefore, 
$[(1-\xi_{ij})A + \xi_{ij}I]z$
is a multiple of the vector $\uno$, 
regardless of $z\in\R^n$.
Since $\xi_{ij}$ depends only on $i,j$ but not on $z$, 
we must conclude that
$(1-\xi_{ij})A + \xi_{ij}I$ is a rank-one matrix
whose image is $\mathrm{Span}(\uno)$ for any $i\neq j$,
that is,
$A = \uno v^T$ for some $v\in\R^n$.
Moreover, $0 = \xi_{ij} = B_{ii} - B_{ji}$, whence
$B = \uno w^T$ for some $w\in\R^n$.
Finally, 
imposing that $\Delta$ is the identity on $\mathcal{D}$ we obtain
$$
   x\uno^T - \uno y^T = \Delta(x\uno^T - \uno y^T) = 
   \uno v^T(x\uno^T - \uno y^T) + (I-\uno v^T)
   (x\uno^T - \uno y^T) w\uno^T ,
$$
which must be true for any choice of $x,y\in\R^n$. 
Using some algebra, we eventually arrive at the identities
$v^T\uno = 1$ and $w^T\uno = 1$, which proves one part of the claim.

The converse implication is simpler to prove.
Indeed, if $M$ and $N$ are as in the  hypotheses then
$$
   \Delta(X) = \uno v^T X
   + Xw\uno^T -
   \uno v^TX w\uno^T
   = x\uno^T - \uno y^T \in\D ,
$$
where we set $x = Xw$ and 
$y = \uno w^T X^T v - X^T v$.
Furthermore, if $X = x\uno^T - \uno y^T$ then,
after some simplification,
\begin{align*}
   \Delta(X) & = \uno v^T (x\uno^T - \uno y^T)
   + (x\uno^T - \uno y^T)w\uno^T -
   \uno v^T(x\uno^T - \uno y^T)w\uno^T \\
   & = \ldots = x\uno^T - \uno y^T = X ,
\end{align*}
and the proof is complete. 
\end{proof}

As we will show shortly after, Theorem \ref{thm:vw} allows us to generalize 
Algorithms \ref{alg:LL1} and \ref{alg:LL2} employing any one of the projectors
described there and recovering the Cauchy points
from one row and column
of $\Delta(A^{[-1]})$.
To this end, we introduce the following notation. Let $u\in\mathbb{R}^n$ and define the matrix
\begin{equation}   \label{eq:M_u}
M_u=I - \uno u^T\in\mathbb{R}^{n\times n}.
\end{equation}

\begin{theorem}   \label{thm:1+2}
Let $A\in\Rnn$ be a matrix with no zero entries.
\begin{itemize}
    \item Let $\Delta_1(X) = X - M_{e_1}XM_{e_1}^T$ and let $x,y$ be the vectors computed by Algorithm \ref{alg:LL1}
with input $A$. Then $H(x,y) = \Delta_1(\hinv{A})$.
\item 
Let  $\Delta_2(X) = X - M_{\uno/n}XM_{\uno/n}^T$. Then $\Delta_2$ is the projector
onto $\mathcal{D}$ orthogonal with respect to the Frobenius inner product
$\langle X,Y\rangle = \mathrm{trace}(Y^TX)$.
Moreover, if $x,y$ are the vectors computed by Algorithm \ref{alg:LL2}
with input $A$ then $H(x,y) = \Delta_2(\hinv{A})$.
\end{itemize}
\end{theorem}

\begin{proof}
Let $A$, $x$ and $y$ be as in the first part of the claim.
Let $Z = \hinv{A}$
and $B = \Delta_1(Z)$ for notational convenience.
With simple passages,
$$
   B = Z - (I - \uno e_1^T)Z(I - e_1\uno^T) 
   = \uno e_1^TZ + Ze_1\uno^T - Z_{11}\uno\uno^T .
$$
Recall that the first row and column of $H(x,y)$ coincide with those of $Z$. From the relations $y_j = -Z_{1j}$ and $x_i = Z_{i1} - Z_{11}$,
for $i,j=1\ldots,n$ we have
$$
   B_{ij} = Z_{1j} + Z_{i1} - Z_{11} = x_i - y_j.
$$
Thus $B = H(x,y)$.
In particular, the first row and column of $B$ 
coincide with those of $\hinv{A}$. 
This proves the first claim.

For the second part of the claim, let $D = x\uno^T - \uno y^T \in\D$ be arbitrary. Then
the identity $\Delta_2(D) = D$ can be derived by elementary manipulations.
Finally, consider the inner product
$\langle X - \Delta_2(X),D\rangle$ for arbitrary $X\in\Rnn$ and $D\in\D$.
We have 
\begin{align*}
   \langle X - \Delta_2(X),D\rangle 
   & = \frac1n \langle \uno\uno^TX + X \uno\uno^T - 
   \uno\uno^TX\uno\uno^T /n , D \rangle \\
   & = \frac1n \langle X , \uno\uno^TD + D \uno\uno^T
   - \uno\uno^T D \uno\uno^T/n \rangle \\
   & = \langle X , D - \Delta_2(D) \rangle = 0 .
\end{align*}
Thus the residual $X - \Delta_2(X)$ is orthogonal to $\D$,
proving that $\Delta_2$ is an orthogonal projector onto $\D$.
In particular, $\Delta_2(\hinv{A})$ attains the minimum in \eqref{eq:minF}, exactly as the output of Algorithm \ref{alg:LL2}, thus proving the last part of the claim.
\end{proof}

Theorem \ref{thm:1+2} characterizes the result of Algorithm \ref{alg:LL1} in terms of the 
displacement operator 
$\Delta_1(X) = X - M_{e_1}XM_{e_1}^T$, which is a projector onto $\D$. 
Also
Algorithm \ref{alg:LL2}, which computes the matrix in $\D$ that is the closest in Frobenius norm to a given matrix, is described in terms of 
the projector $\Delta_2$.
It is then natural to ask if we can devise other algorithms for reconstructing a Cauchy approximation of $A$ that correspond to different 
projectors onto $\D$.
Also, recall that Algorithm \ref{alg:LL1} only uses information from the first row and column of $A$, thus achieving linear complexity, whereas it might be useful to use information from the other matrix elements as well, even at an increased computational cost.

Algorithm \ref{alg:generaluv} here below
computes the 
normalized generators of the matrix obtained 
from a generic projector from Theorem \ref{thm:vw}.
The vectors $v,w\in\R^n$ that characterize the projector
are given in input to the algorithm, together with the data matrix $A$.
The correctness of the algorithm is shown in 
Lemma \ref{lem:example}, and {\it a priori} bounds on the 
approximation error are given in Theorem \ref{thm:kappastargen}.
The computational cost is, in general, $O(n^2)$, but the algorithm is well-suited for parallel implementation, with a cost of $O(n)$ per processor.

\begin{algorithm}[ht]   \label{alg:generaluv}
\LinesNumbered 
\SetNlSty{texttt}{}{.} 
\DontPrintSemicolon
\caption{Recovery of normalized generators -- variant with parametrized projectors}   
\KwIn{Matrix $A = (A_{ij})$ with no zero entries, vectors $v$ and $w$ such that $v^T\uno=1$, $w^T\uno=1$}
\KwOut{Normalized generators $x,y$ }
\BlankLine
\For{$i = 1$ \KwTo $n$}
   {
   $y_i = -\sum_{k=1}^n v_k/A_{ki}$ \;  
   }
$\theta=\sum_{k=1}^n y_k w_k$\;
\For{$i = 1$ \KwTo $n$}{
$x_{i} = \theta + \sum_{k=1}^n w_k/A_{ik}$ \;
}
$\alpha = (\sum_{i=1}^n x_i + y_i)/(2n)$\;
$x = x - \alpha\uno$\;
$y = y - \alpha\uno$\;
\end{algorithm}

\begin{lemma}   \label{lem:example}
Let $v,w\in\mathbb{R}^n$ such that $v^T\uno=w^T\uno =1$.
Let $A\in\Rnn$ be a matrix with no zero entries,
and let $x,y$ be the vectors computed by Algorithm \ref{alg:generaluv}
with input $A$. Then 
$H(x,y) = \hinv{A} - M_v\hinv{A}M_w^T$.
In particular, if $A$ is Cauchy then $A = \cau(x,y)$.
\end{lemma}

\begin{proof}
From the definition of $M_v$ and $M_w$ we have
\begin{align*}
   \hinv{A} - M_v\hinv{A}M_w^T 
   & = \hinv{A} - (I-\uno v^T)\hinv{A}(I- w \uno^T) \\
   & = \uno v^T \hinv{A} + \hinv{A}w\uno^T - \uno v^T\hinv{A}w\uno^T.
\end{align*}
The entry in position $(i,j)$ is given by
\begin{align*}
   [\hinv{A} - M_v\hinv{A}M_w^T]_{i,j} 
   & = e_i^T (\hinv{A} - M_v\hinv{A}M_w^T) e_j \\
   & = e_i^T\uno v^T \hinv{A}e_j + e_i^T\hinv{A}w\uno^T e_j - e_i^T\uno v^T\hinv{A}w\uno^T e_j  \\
   & = e_i^T\hinv{A}w - v^T\hinv{A}w 
   + v^T \hinv{A}e_j.
\end{align*}
Now, let
$$ 
   y = - {\hinv{A}}^T v 
   , \qquad 
   \theta = - v^T \hinv{A} w , \qquad
   x = \theta \uno + \hinv{A} w . 
$$
Note that these quantities are those computed in lines 2, 3, and 5 of Algorithm \ref{alg:generaluv}, respectively.
Therefore we have
$$
[\hinv{A} - M_v\hinv{A}M_w^T]_{i,j} = x_i - y_j,
$$
with $x_i$, $y_j$ computed as in Algorithm \ref{alg:generaluv}.
The last assertion follows immediately from the property $M_v\uno = M_w\uno = 0$.     
\end{proof}

Hereafter we make use of the best approximation measure
$$
   \mu_\star(A) = \min_{X\in \mathcal{D}}\|\hinv{A}-X\|_\star ,
$$
where $\|\,\cdot\,\|_\star$ stands for either the Chebyshev norm
or the Frobenius norm, according to whether $\star = \che$ or $\star = \fro$, respectively.
Recall that Algorithm \ref{alg:LL2} computes the solution 
of \eqref{eq:minF}, hence
$\mu_\fro(A)$ can be obtained explicitly from that solution.
For later reference, we state in the next lemma a permutational invariance property
of this measure, whose trivial proof is omitted for brevity.

\begin{lemma}   \label{lem:permut}
Let $P,Q$ be permutation matrices. If $A$ has only nonzero entries then
$\mu_\star(PAQ) = \mu_\star(A)$.
\end{lemma}

The next result shows that the matrix $H(x,y)$ 
obtained from Algorithm \ref{alg:generaluv} is never too far from a matrix in $\mathcal{D}$ that is closest to $\hinv{A}$,
in both the Frobenius and Chebyshev norms. Recall that Algorithm \ref{alg:generaluv}
includes Algorithm \ref{alg:LL1} and Algorithm \ref{alg:LL2} as particular cases.

\begin{theorem}   \label{thm:kappastargen}
Let $H(x,y)\in\mathcal{D}$ be the matrix 
obtained from the output of Algorithm \ref{alg:generaluv},
that is, $H(x,y) = \hinv{A} - M_v\hinv{A}M_w^T$, where $A$ has no zero entries.
Then,
\begin{equation}    \label{eq:bestbound}
   \|\hinv{A} - H(x,y)\|_\star \leq \alpha \mu_\star(A) 
\end{equation}
with $\alpha = \|M_v\|_\infty\| M_w\|_\infty$
if $\star = \max$ and $\alpha = \|M_v\|_2\| M_w\|_2$ if $\star = \fro$.
Furthermore, if $x$ and $y$ are Cauchy points then, for $C = \cau(x,y)$ we also have
\begin{equation}   \label{eq:thm3.5}
   \frac{\|A - C\|_\star}{\|A\|_\star} 
   \leq \nu \mu_\che(A)\|C\|_\che , \qquad
   \frac{\|A - C\|_\star}{\|C\|_\star} 
   \leq \nu \mu_\che(A)\|A\|_\che .
\end{equation}
with $\nu = \|M_v\|_\infty\| M_w\|_\infty$.
\end{theorem}

\begin{proof}
Let $Z = \hinv{A}$ and let $B\in\mathcal{D}$ be a matrix
such that $\mu_\star(A) = \|\hinv{A} - B\|_\star$.
Consider first the Chebyshev norm case, $\star = \max$. 
By Lemma \ref{lem:example} we know that $M_v B M_w^T$
is the zero matrix. Hence,
\begin{align*}
   \|Z - H(x,y)\|_\che = \|M_v Z M_w^T\|_\che 
   & = \|M_v (Z - B) M_w^T\|_\che  \\  
   & = \|(M_w\otimes M_v) \vec(Z - B) \|_\infty \\
   & \leq \|M_w\otimes M_v\|_\infty \|Z - B\|_\che .
\end{align*}
The identity $\|M_w\otimes M_v\|_\infty = \|M_w\|_\infty\| M_v\|_\infty$
completes the proof.
The Frobenius norm case goes exactly along the same lines but 
making use of the bound
$$
   \|(M_w\otimes M_v)\vec(Z - B)\|_2 \leq 
   \|M_w\otimes M_v\|_2 \|Z - B\|_\fro 
$$
and the identity $\|M_w\otimes M_v\|_2 = \|M_w\|_2\| M_v\|_2$.

Finally, assuming that $x$ and $y$ are Cauchy points, let $B = H(x,y)$ so that $C = \hinv{B}$.
First we note that, for $i,j = 1,\ldots,n$ we have
\begin{align*}
   \bigg|\frac{1}{A_{ij}} - B_{ij}\bigg| =
   \bigg|\frac{1}{A_{ij}} - \frac{1}{C_{ij}}\bigg| & =
   \frac{|A_{ij} - C_{ij}|}{|A_{ij}C_{ij}|} \\
   & \geq \frac{|A_{ij} - C_{ij}|}{|A_{ij}|\|C\|_\che} .
\end{align*}
Let $\nu = \|M_v\|_\infty\| M_w\|_\infty$. Using the first part, we get
\begin{align*}
   \frac{|A_{ij} - C_{ij}|}{|A_{ij}|} 
   & \leq
   \bigg|\frac{1}{A_{ij}} - B_{ij}\bigg| \|C\|_\che \\
   & \leq \|\hinv{A} - B\|_\che\|C\|_\che
   \leq \nu \mu_\che(A)\|C\|_\che .
\end{align*}
Lemma \ref{lem:easy} now applies to prove the leftmost inequality
in \eqref{eq:thm3.5}.
Analogously,
$$
   \bigg|\frac{1}{A_{ij}} - B_{ij}\bigg| =
   \frac{|A_{ij} - C_{ij}|}{|A_{ij}C_{ij}|}
   \geq \frac{|A_{ij} - C_{ij}|}{|C_{ij}|\|A\|_\che}  .
$$
Hence 
$$
   \frac{|A_{ij} - C_{ij}|}{|C_{ij}|} \leq 
   \bigg|\frac{1}{A_{ij}} - B_{ij}\bigg| \|A\|_\che 
   \leq \nu \mu_\che(A) \|A\|_\che .
$$
It remains to apply Lemma \ref{lem:easy} to complete the proof.
\end{proof}

The next lemma provides a technical result required later.

\begin{lemma}   \label{lem:M_u}
For any $u\in\R^n$ we have 
\begin{enumerate}
\item $\|M_u\|_\infty \leq 1 + \|u\|_1$,
with equality if $u_i > 0$ for some $i=1,\ldots,n$;
\item $\|M_u\|_2 \leq 1 + \sqrt{n}\|u - \uno/n\|_2$.
\end{enumerate}
\end{lemma}

\begin{proof}
From \eqref{eq:M_u} we 
obtain immediately $\|M_u\|_\infty \leq \|I\|_\infty + \|\uno u^T\|_\infty = 1 + \|u\|_1$. 
Moreover, if $u_i > 0$ then 
$\|M_u\|_{\infty} \geq \sum_j |M_{ij}| = 1 + \sum_j |u_j| = 1+\|u\|_1$, 
and the first part follows.
To prove the second part, we first observe that 
$I-\uno\uno^T/n$ is the orthogonal projector onto $\mathrm{Span}(\uno)$, 
so $\|I - \uno\uno^T/n\|_2 = 1$. Thus,
\begin{align*}
   \|M_u\|_2 = \|I-\uno u^T\|_2 
   & \leq \|I - \uno\uno^T/n\|_2 + \|\uno(u - \uno/n)^T \|_2 \\
   & = 1 + \|\uno\|_2  \|u - \uno/n\|_2 
   = 1 + \sqrt{n}\|u - \uno/n\|_2 ,
\end{align*}
and the proof is complete.
\end{proof}

The next corollary provides easily computable formulas
for the constant $\alpha$ in Theorem \ref{thm:kappastargen}.
We refrain from including the proof since it is an immediate
consequence of  
Lemma \ref{lem:M_u}

\begin{corollary}   \label{cor:kappastargen}
For any vectors $v,w\in\R^n$,
the constant $\alpha$ in \eqref{eq:bestbound}
can be chosen as
\begin{enumerate}
\item $(1 + \| v \|_1)(1  + \| w \|_1)$ if $\star = \max$,
with equality if $v_i > 0$ and $w_j > 0$ for some $i,j=1,\ldots,n$;
\item $(1 + \sqrt{n}\| v - \uno/n\|_2)
      (1  + \sqrt{n}\| w - \uno/n \|_2)$ if $\star = \fro$.
\end{enumerate}
\end{corollary}

We point out that the inequalities in \eqref{eq:bestbound}
and in Corollary \ref{cor:kappastargen} can be attained as equalities. For instance, we know that 
when $x$ and $y$ are computed by Algorithm \ref{alg:LL2}
we have $\|\hinv{A}-H(x,y)\|_\fro = \mu_\fro(A)$ if $A$ has no zero entries, owing to 
Theorem \ref{thm:1+2}.
On the other hand, point 2 of Corollary~\ref{cor:kappastargen}
gives $\alpha = 1$ as $v = w = \uno/n$ in Algorithm \ref{alg:LL2}, showing that the estimate in Theorem~\ref{thm:kappastargen} is optimal.
The result below provides analogous error bounds 
for the approximation given by Algorithm \ref{alg:LL1}.

\begin{corollary}   \label{cor:LL1}
Let $H(x,y)\in\mathcal{D}$ be the matrix 
obtained from the output of Algorithm \ref{alg:LL1},
that is, $H(x,y) = \hinv{A} - M_{e_1}\hinv{A}M_{e_1}^T$.
Then,
$\|\hinv{A} - H(x,y)\|_\star \leq \alpha \mu_\star(A)$
where $\alpha = 4$
if $\star = \max$ and $\alpha = n$ if $\star = \fro$.
Furthermore, if $x$ and $y$ are Cauchy points then, for $C = \cau(x,y)$ we also have
$$
   \frac{\|A - C\|_\star}{\|A\|_\star} 
   \leq 4\mu_\che(A)\|C\|_\che , \qquad
   \frac{\|A - C\|_\star}{\|C\|_\star} 
   \leq 4\mu_\che(A)\|A\|_\che .
$$
\end{corollary}

\begin{proof}
The identity $\|M_{e_1}\|_\infty = 2$ is an immediate consequence of 
Corollary \ref{cor:kappastargen}, point 1.
To compute $\|M_{e_1}\|_2$, note that
$$
   M_{e_1}M_{e_1}^T = (I-\uno e_1^T)(I - e_1 \uno^T) = 
   I - \uno e_1^T - e_1 \uno^T + \uno\uno^T = 
   \begin{pmatrix} 0 & \\ & S \end{pmatrix} ,
$$
where $S = I + \uno\uno^T\in\R^{(n-1)\times (n-1)}$. 
By elementary techniques and Perron-Frobenius theory,
we deduce that the spectral radius of $S$
is $n$, due to the identity $S\uno = n\uno$.
Consequently,
$\|M_{e_1}\|_2 = \sqrt{\rho(M_{e_1}M_{e_1}^T)} = \sqrt{\rho(S)} = \sqrt{n}$.
The claim is now a consequence of Theorem \ref{thm:kappastargen}.
\end{proof}

\begin{remark}
    In view of Theorem \ref{thm:kappastargen}, one may want to choose vectors $v$ and $w$ in Algorithm \ref{alg:generaluv} to minimize $\alpha$. 
    Because of the definition of $\mu_\star$, we trivially have $\alpha \geq 1$, and this lower bound is attained in the Frobenius norm case when $v = w = \uno/n$ as in Algorithm \ref{alg:LL2}, see Corollary  \ref{cor:kappastargen}. Instead, if $\star = \max$ then the lower bound for $\alpha$ increases to $4$. Indeed, the constraints $v^T\uno=w^T\uno=1$ imply that $v$ and $w$ have at least one positive entry, and moreover, $\|v\|_1\geq 1$ and $\|w\|_1\geq 1$. Hence, $\alpha = \|M_v\|_\infty\|M_w\|_\infty \geq 4$ by Lemma \ref{lem:M_u}. Choosing $v$ and $w$ with nonnegative entries we have $\|v\|_1=\|w\|_1=1$ and therefore $\alpha = \|M_v\|_\infty\|M_w\|_\infty = 4$.
    This value is attained when, for example, $v = w = e_1$ as in Algorithm \ref{alg:LL1}, or $v = w = \uno/n$ as in Algorithm \ref{alg:LL2},
    see Theorem \ref{thm:1+2}.
\end{remark}

\begin{remark} 
The entrywise reciprocal of a Cauchy matrix $A$,  $A^{[-1]}$, belongs to the set $\mathcal{D}$.
The linear space $\mathcal{D}$ represents a linearly parameterized family of matrices, which relates to the general structured matrix recovery framework developed in \cite{amsel26} and \cite{Halikias2024}. Unlike the approach of \cite{amsel26,Halikias2024}, which uses matrix-vector products to recover the representation of a matrix, 
 our algorithms operate directly on the entries of  $A^{[-1]}$. It would be interesting to explore a possible connection between the two approaches, which might lead to a new interpretation in terms of oversampling.   
\end{remark}


\section{A deeper analysis of Algorithm~\ref{alg:LL1} using CUR approximation theory}   \label{sec:CUR}

Due to its simplicity and the effectiveness of
the error bounds in Corollary \ref{cor:LL1},
the possibility of using
Algorithm~\ref{alg:LL1} to recover a perturbed Cauchy matrix deserves a deeper analysis.
In what follows, we are concerned with \emph{a priori} error bounds, 
that is, bounds that are inherent to Algorithm \ref{alg:LL1}.
In fact, Theorem \ref{thm:LL1} allows us to estimate the error $\|A-\hinv{D}\|_\star$ on the basis of the entrywise residual $|1-A_{ij}D_{ij}|$, thus giving an `a posteriori' bound. Instead, in this section we present estimates of the approximation error based on more intrinsic properties of $A$ or other closely related matrices.
Furthermore, while the hypothesis \eqref{eq:hyp} constrains the entries
of $D$ (and $C$, as a consequence) to agree in sign with the corresponding entries of $A$, the results in this section are much less restrictive. 
Actually, Corollary \ref{cor:LL1} is one result in this vein, which follows from the projector-based analysis developed in Section
\ref{sec:generalfamily}. Here, we present a further study of Algorithm \ref{alg:LL1} based on CUR approximation theory.

The CUR approximation of a matrix is a technique employed in numerical linear algebra to devise low-rank approximations of matrices \cite{MR2443975,MR4068947}. The concept revolves around decomposing a given matrix $A$ into three factors, 
often denoted by $C$, $U$ and $R$.
Here, $C$ is a matrix formed by a selection of columns from $A$, $R$ denotes a selection of rows from $A$, and  $U$ is a small matrix linking the previous two. This may result in the approximated factorization $A\approx CUR$, which may become an identity when the rank of the factors equals the rank of $A$. In what follows, we make use of the following result, which 
is a minor reworking of Theorem 2.2 in \cite{GoTy01}, one of the main results in CUR approximation theory.

\begin{theorem}   \label{thm:CUR}
    Suppose that $A$ is a block matrix of the form
    $$
       A = \begin{pmatrix} A_{11} & A_{12} \\ A_{21} & A_{22} \end{pmatrix} ,
    $$
    where $A_{11}$ is $k\times k$ and nonsingular, 
    and, if $B$ is any $k\times k$ submatrix of $A$ then $|\det(A_{11})|\geq \nu |\det(B)|$
    for some $0 < \nu \leq 1$. 
    Let $C = [A_{11} ; A_{21}]$, $U = A_{11}^{-1}$,
    and $R = [A_{11} \ A_{12}]$.
    Then
    $$
       \|A - CUR\|_\che
       \leq \nu^{-1}(k+1) \sigma_{k+1}(A) ,
    $$
    where $\sigma_{k+1}(A)$ is the $(k+1)$-th largest singular value of $A$.
\end{theorem}

Let $A$ be a possibly perturbed Cauchy matrix with no zero entries.
Without loss of generality, we can assume that $\max_{i,j}|A_{ij}|=1$.
Indeed, the reconstruction of a Cauchy matrix is 
homogeneous with respect to
multiplicative constants. For clarity, if $x$ and $y$ are the vectors
computed by Algorithm \ref{alg:LL1} with input matrix $A$
and we let $C = \cau(x,y)$ then, for any nonzero scalar $\alpha$, 
the vectors
computed by Algorithm \ref{alg:LL1} with input matrix $\alpha A$
are $x/\alpha$ and $y/\alpha$, so that the reconstructed Cauchy matrix 
is $\cau(x/\alpha,y/\alpha) = \alpha C$.
Furthermore, by Lemma \ref{lem:permut},
we can safely suppose that $A_{11} = 1$ is an entry of maximum modulus.

For notation simplicity,
let $Z = \hinv{A}$. Partition the matrix $Z$ as follows:
$$
   Z = \begin{pmatrix} 1 & w^T \\ v & Y \end{pmatrix} 
$$
where $v,w\in\R^{n-1}$ and $Y\in\R^{(n-1)\times (n-1)}$.
Let $B = Z - M_{e_1}ZM_{e_1}^T$ be the matrix constructed from the output of
Algorithm 1. By simple computations,  
$$
   Z - B = \begin{pmatrix} 0 & 0 \\ 0 & S \end{pmatrix} 
$$
where $S = Y - \uno w^T - v\uno^T + \uno\uno^T$.
Introduce the bordered matrix
\begin{equation}   \label{eq:Zhat}
   \widehat Z = \begin{pmatrix}
   0 & \uno^T \\
   \uno & Z \end{pmatrix}
   = \begin{pmatrix}
   0 & 1 & \uno^T \\
   1 & 1 & w^T \\ 
   \uno & v & Y \end{pmatrix} .
\end{equation}

Below, we prove that 
the residual $\hinv{A} - H(x,y)$
of Algorithm~\ref{alg:LL1}
can be written as the Schur complement of 
the $2\times 2$ leading minor of the matrix in \eqref{eq:Zhat}.

\begin{lemma}
Let $H(x,y)$ be the matrix obtained by the output of
Algorithm~\ref{alg:LL1} from $Z = \hinv{A}$.
Then, $\|\hinv{A} - H(x,y)\|_\star$
is equal to the residual norm of the 
rank-2 CUR approximation computed from the first two rows and columns of $\widehat Z$.
\end{lemma}

\begin{proof}
Consider the partitioning
$$ 
   \widehat Z = \begin{pmatrix}
   \widehat Z_{11} & \widehat Z_{12} \\
   \widehat Z_{21} & Y \end{pmatrix} 
$$ 
of the matrix \eqref{eq:Zhat}.
Note that $\wh Z_{11}\in\R^{2\times 2}$,
$\wh Z_{12} = (\uno \ w)^T \in\R^{2\times(n-1)}$ and
$\wh Z_{21} = (\uno \ v) \in\R^{(n-1)\times 2}$.
In particular,
$$
   \wh Z_{11} = \begin{pmatrix} 0 & 1\\ 1 & 1 \end{pmatrix},
   \qquad
   \wh Z_{11}^{-1} = \begin{pmatrix} -1 & 1\\ 1 & 0 \end{pmatrix} .
$$
The Schur complement of $\wh Z_{11}$ in $\wh Z$ is
$$
   Y - \wh Z_{21}\wh Z_{11}^{-1}\wh Z_{12}
   = Y - v\uno^T + \uno\uno^T - \uno w^T = S .
$$
Let $B = H(x,y)$. We have
\begin{align*}
   \wh Z -  
   \begin{pmatrix} \wh Z_{11} \\ \wh Z_{21} \end{pmatrix}
   \wh Z_{11}^{-1} 
   \begin{pmatrix} \wh Z_{11} & \wh Z_{12} \end{pmatrix}
   & = \wh Z - 
   \begin{pmatrix}
   0 & 1 & \uno^T \\
   1 & 1 & w^T \\
   \uno & v & Y - S \end{pmatrix} \\
   & = \begin{pmatrix}
   0 & 0 & 0 \\
   0 & 0 & 0 \\
   0 & 0 & S  \end{pmatrix} = 
   \begin{pmatrix}
   O &  \\
     & Z - B \end{pmatrix}.
\end{align*}
Thus $\|\hinv{A} - H(x,y)\|_\star$ is equal to the norm of the latter matrix.
\end{proof}

Finally, $\|\hinv{A} - H(x,y)\|_\star$
can be estimated using
Theorem \ref{thm:CUR}, as follows.

\begin{corollary}   \label{cor:CUR}
Under the (non-restrictive) assumptions on $A$ and $Z$ stated above,
if $x,y$ is the output of Algorithm 1 applied to $A$ and $\wh Z$ is the matrix in \eqref{eq:Zhat} then
$$
   \|\hinv{A} - H(x,y)\|_\che \leq 
   6 \sigma_{3}(\wh Z) .
$$ 
\end{corollary}

\begin{proof}
By \eqref{eq:Zhat}, we have $|\det(\wh Z_{11})| = 1$. Furthermore, since $|A_{ij}| \leq 1$ for all $i,j$, the determinant of every $2\times 2$
submatrix of $A$ has a modulus not larger than $2$.
Thus, the claim follows straightforwardly from  
the previous lemma and Theorem \ref{thm:CUR} applied to $\wh Z$ with $k = 2$ and $\nu = 1/2$.
\end{proof}

The inequality in Corollary \ref{cor:CUR}
shows a nontrivial
relation between the reconstruction error in max-norm
and the third singular value of a bordering 
of $\hinv{A}$, thus providing a guarantee 
on the quality of the output of Algorithm~\ref{alg:LL1} that can be estimated a priori from $A$. 
The following result shows that $\sigma_3(\wh Z)$ 
quantifies the (non-)Cauchyness of $A$,
being essentially equivalent to $\mu_\star(A)$.

\begin{theorem}
Let $A$ be a matrix with no null entries, and let $\wh Z$ be the matrix in \eqref{eq:Zhat}.
Then
$$
   \mu_\che(A)/6 \leq \sigma_3(\wh Z) \leq \mu_\fro(A) .
$$
In particular, $A$ is Cauchy if and only if $\sigma_3(\wh Z) = 0$.
\end{theorem}

\begin{proof}
In the notation of Corollary \ref{cor:CUR}, 
we have $$\mu_\che(A) \leq \|\hinv{A} - H(x,y)\|_\che \leq 6 \sigma_{3}(\wh Z),$$ and the leftmost inequality in the claim follows.
For the other inequality, let $B\in\mathcal{D}$
be the matrix such that $\|\hinv{A} - B\|_\fro = \mu_\fro(A)$. Introduce the bordered matrix
$$
   \wh B = \begin{pmatrix} 0 & \uno^T \\
   \uno & B \end{pmatrix} .
$$
Note that $\wh B$ has rank $2$. Indeed,
$B = H(x,y)$ for some vector pair $(x,y)$, and
$$
   \wh B = \begin{pmatrix} 0 & \uno^T \\ \uno & 
   x\uno^T - \uno y^T \end{pmatrix}
   = \begin{pmatrix} 0 & 1 \\ \uno & x \end{pmatrix}
   \begin{pmatrix} 1 & 0 \\ -y & \uno \end{pmatrix}^T .
$$
Then, by Eckart-Young theorem,
$$
   \mu_\fro(A) = \| \hinv{A} - B\|_\fro 
   = \| \wh Z - \wh B\|_\fro \geq \sigma_3(\wh Z) ,
$$
and the proof is complete.
\end{proof}


\section{A Sylvester-type displacement-based recovery method}
\label{sec:displacementbased}

In this section, we introduce a new optimality criterion 
to approximate a given matrix $A$ by a Cauchy matrix.
This criterion leads to an $O(n^3)$ algorithm to 
compute a Cauchy approximation to $A$ that does not make use of $\hinv{A}$.
Given vectors $x,y$ the matrix operator 
$\nabla_{x,y}:A\mapsto \mathrm{Diag}(x) A - A \,\mathrm{Diag}(y)$ is often called a Sylvester
displacement operator \cite{MR3681370,MR533501}. This operator is invertible if and only if $x$ and $y$ are Cauchy points, and  in this case, the matrix $\cau(x,y)$ is the (unique) solution of 
the Sylvester matrix equation $\nabla_{x,y}(X) = \uno\uno^T$.
Thus we may consider the number
\begin{equation}   \label{eq:defbeta}
   \beta_\star(A) = \min_{x,y}\|\nabla_{x,y}(A) - \uno\uno^T\|_\star
\end{equation}
as a measure of `Cauchyness' of a matrix $A$.
Indeed, $\beta_\star(A) = 0$ if and only if $A$ is a Cauchy matrix.
Moreover, Theorem \ref{thm:LL1}
proved that 
if $\beta_{\max}(A)$
is less than $1$ then $A$ is close to a Cauchy matrix.
The relationship between $\beta_\star$ and $\mu_\star$
is shown below.

\begin{theorem}   \label{thm:betakappa}
If $A$ is a matrix without null entries, then
$$
    \frac{\mu_\star(A)}{\|\hinv{A}\|_\che} 
    \leq \beta_\star(A) \leq \mu_\star(A) \|A\|_\che .
$$
\end{theorem}

\begin{proof}
For any two vectors $v,w\in\R^n$, 
vectorizing the matrix $\nabla_{v,w}(A) - \uno\uno^T$ via \eqref{eq:ABC} we obtain
\begin{align*}
   \vec(\nabla_{v,w}(A) - \uno\uno^T) & = 
   (\mathrm{Diag}(v) \otimes I - I \otimes \mathrm{Diag}(w))\vec(A) - \uno \\
   & = \mathrm{Diag}(\vec(A))
   \big[v\otimes \uno - \uno\otimes w - \vec(\hinv{A}) \big] \\
   & = \mathrm{Diag}(\vec(A))
   \big[\vec\big(H(v,w) - \hinv{A}\big) \big] .
\end{align*}
Hence, if $v,w$ are such that $\mu_\star(A) = \|\hinv{A} - H(v,w)\|_\star$
then
$$
   \beta_\star(A) \leq \|\nabla_{v,w}(A) - \uno\uno^T\|_\star
   \leq \|A\|_{\max} \|H(v,w) - \hinv{A}\|_\star 
   = \|A\|_{\max} \mu_\star(A) ,
$$
which proves the rightmost inequality in the claim.
With similar arguments,
$$
   \vec\big(H(x,y) - \hinv{A}\big) =
   \mathrm{Diag}(\vec(A))^{-1} \vec(\nabla_{x,y}(A) - \uno\uno^T) .
$$
Noting that $\mathrm{Diag}(\vec(A))^{-1} = \mathrm{Diag}\big(\vec(\hinv{A})\big)$
and taking norms, we get
$$
   \|H(x,y) - \hinv{A}\|_\star \leq 
   \|\hinv{A}\|_\che \|\nabla_{x,y}(A) - \uno\uno^T\|_\star .
$$
Take $v,w$ such that $\beta_\star(A) = \|\nabla_{v,w}(A) - \uno\uno\|_\star$.
Then,
\begin{align*}
   \beta_\star(A) = \|\nabla_{v,w}(A) - \uno\uno\|_\star
   & \geq \|H(v,w) - \hinv{A}\|_\star / \|\hinv{A}\|_\che \\
   & \geq \min_{x,y}\|H(x,y) - \hinv{A}\|_\star / \|\hinv{A}\|_\che\\
   & = \mu_\star(A) / \|\hinv{A}\|_\che ,
\end{align*}
and the proof is complete.
\end{proof}

The forthcoming argument proves that 
$\beta_\star(A)$ is related to the distance of $A$ from the set of Cauchy matrices.

\begin{proposition}
Let $A$ be a perturbed Cauchy matrix.
If $x,y$ are Cauchy points such that 
$\beta_\star(A) = \|\nabla_{x,y}(A) - \uno\uno^T\|_\star$ then
$$
     \frac{\beta_\star(A)}{\max|x_i - y_j|} 
     \leq \|A - \cau(x,y)\|_\star
     \leq \frac{\beta_\star(A)}{\min|x_i - y_j|}.
$$
\end{proposition}

\begin{proof}
    Consider the operator norm
$$
   \|\nabla_{x,y}\|_\star
   = \sup_{X\in\R^{n\times n}, X\neq O}
   \frac{\|\nabla_{x,y}(X)\|_\star}{\|X\|_\star} .
$$
Owing to the identity
$$
   \vec(\nabla_{x,y}(X)) = 
   (\mathrm{Diag}(x) \otimes I - I\otimes
   \mathrm{Diag}(y)) \vec(X) ,
$$
it is not difficult to see that 
$\|\nabla_{x,y}\|_\star = \max_{i,j}|x_i-y_j|$.
Furthermore, since $x,y$ are Cauchy points then
$\nabla_{x,y}$ is invertible and we have
$\|\nabla_{x,y}^{-1}\|_\star = 1/\min_{i,j}|x_i-y_j|$.
Let $E = A - \cau(x,y)$. 
Then, $\|\nabla_{x,y}(A) - \uno\uno^T\|_\star = \|\nabla_{x,y}(A - \cau(x,y))\|_\star = \|\nabla_{x,y}(E)\|_\star$. Moreover,
$$
   \|\nabla_{x,y}(E)\|_\star
   \leq \|\nabla_{x,y}\|_\star
   \|E\|_\star  = \max|x_i - y_j| \|E\|_\star .
$$
Analogously, 
$$
   \|E\|_\star = \|\nabla_{x,y}^{-1}(\nabla_{x,y}(E))\|_\star
   \leq \|\nabla_{x,y}^{-1}\|_\star
   \|\nabla_{x,y}(E)\|_\star  
   = \frac{\|\nabla_{x,y}(E)\|_\star}{\min_{i,j}|x_i - y_j|} ,
$$
and we have the claim.
\end{proof}

\subsection{Minimizing the entrywise relative error}   

The constant $\beta_\star(A)$ introduced in \eqref{eq:defbeta}
can be interpreted as a measure of the closeness of a matrix $A$ to the set of Cauchy matrices. In this section, we show that the computation of $\beta_\fro(A)$
can be performed exactly and efficiently by solving a least squares problem. The solution to this problem
immediately provides a Cauchy matrix that is closest to $A$ in 
the sense that we specify hereafter, exactly as 
Algorithm \ref{alg:LL2} yields a Cauchy approximation to a given 
matrix $A$ by computing $\mu_\fro(A)$.

\begin{remark}
Before proceeding further, we point out that the $(i,j)$-entry of the matrix $\nabla_{x,y}(A) - \uno\uno^T$ appearing in the 
definition of $\beta_\star(A)$ is $A_{ij}(x_i-y_j) - 1$.
The identities
$$
   \frac{|1/A_{ij} - (x_i-y_j)|}{1/|A_{ij}|} = 
   \big|A_{ij}(x_i-y_j) - 1\big| 
$$
and
$$
   \frac{|A_{ij} - 1/(x_i-y_j)|}{1/|x_i-y_j|} = 
   \big|A_{ij}(x_i-y_j) - 1\big| 
$$
suggest that the optimization of $\|\nabla_{x,y}(A) - \uno\uno^T\|_\star$ is related to
the minimization of entrywise relative errors 
between $A$ and the Cauchy matrix that approximates it. By comparison, note that  
$\mu_\star(A)$ derives from minimizing the entrywise absolute error $|1/A_{ij} - (x_i-y_j)|$.
\end{remark}

When $\star = \fro$ then \eqref{eq:defbeta}
can be restated as follows:
$$
   \min_{x,y\in\R^n} \sum_{i,j} 
   (A_{ij}(x_i - y_j) - 1)^2 .
$$
This problem can be written in the
form of a linear least squares problem as follows:
\begin{equation}   \label{eq:betaF}  
   \min_{x,y} \| WU 
   \begin{bmatrix} x \\ y \end{bmatrix} 
   - \uno \|_2^2
\end{equation}
where $W = \mathrm{Diag}(\vec(A))$
and $U$ is the $n^2\times 2n$ matrix
$$
   U = \big[ \uno\otimes I \quad -I\otimes \uno \big] .
$$
Equation \eqref{eq:betaF} admits infinite solutions
corresponding to the different pa\-ra\-me\-trizations
of a matrix in $\D$.
In fact, the rank of $U$ is $2n-1$, and the kernel of $U$
consists of the constant vectors.
So the normalized (i.e., least 2-norm) solution of \eqref{eq:betaF}
is given by
\begin{equation}   \label{eq:WU+1}
   \begin{bmatrix} x \\ y \end{bmatrix} 
   = (WU)^+ \uno ,
\end{equation}
where the superscript ${}^+$ represents the Moore-Penrose inverse.

To proceed, we recall the following result on the Moore–Penrose inverse of a partitioned matrix.
The result is a particular case of Lemma 1 in \cite{MR0387302}.

\begin{lemma}   \label{lem:aux}
    Let $M$
    be a real matrix partitioned as 
    $M = \big[A;B \big]$.
    Then 
    $$ M^+ = \begin{bmatrix} A^+ 
    &  B^T(BB^T)^+ \end{bmatrix}
    $$
    if and only if $AB^T = O$.
\end{lemma}

\begin{theorem}
The normalized solution of \eqref{eq:betaF}
is 
$$
   \begin{bmatrix} x \\ y \end{bmatrix} 
   = \bigg( U^TW^2U + \frac{1}{2n} \uno\uno^T 
   \bigg)^{-1}U^T W \uno .
$$
\end{theorem}

\begin{proof}
The matrix $W$ is nonsingular, due to the hypothesis $A_{ij} \neq 0$, and $v = \uno/\sqrt{2n}$ 
is a unit vector in the kernel of $U$. 
Consider the matrix 
$$
   M = \begin{pmatrix} WU \\ v^T \end{pmatrix} 
   \in \R^{(n^2+1) \times 2n} .
$$
Note that $M$ has full column rank.
Thus $M^TM = U^TW^2U + vv^T$ is invertible and 
\begin{align*}
   M^+ = (M^TM)^{-1}M^T
   & = (U^TW^2U + vv^T)^{-1} 
   \begin{bmatrix} U^TW & v \end{bmatrix} \\
   & = \begin{bmatrix} (U^TW^2U + vv^T)^{-1}U^TW 
   & (U^TW^2U + vv^T)^{-1}v \end{bmatrix} .
\end{align*}
On the other hand, since $WUv = 0$, 
from Lemma \ref{lem:aux} we also obtain
$$
   M^+ = \begin{bmatrix} (WU)^+ 
   & v/(v^Tv)  \end{bmatrix} .
$$
The claim follows by equating the first $n^2$ columns of the two partitionings of $M^+$
and using \eqref{eq:WU+1}.
\end{proof}

From the previous theorem,
the unknown parameters can be recovered from the solution of the linear system
\begin{equation}   \label{eq:normaleq}
   \bigg( U^TW^2U + \frac{1}{2n} \uno\uno^T 
   \bigg)
   \begin{bmatrix} x \\ y \end{bmatrix}
   = U^T W \uno .
\end{equation}
The right-hand side of the linear system is
$$
   U^T W \uno = \begin{pmatrix} 
   b^{(1)} \\ b^{(2)} \end{pmatrix} 
$$
with 
$$
    b^{(1)}_{i} = \sum_{j=1}^n A_{ij}  
    \qquad
    b^{(2)}_{i} = \sum_{j=1}^n A_{ji} .
$$
The matrix $U^TW^2U$ has the $2\times 2$ block form
\begin{equation}\label{eq:UTW2U}
   U^TW^2U = \begin{pmatrix} 
   \mathrm{Diag}(d^{(1)}) & -B \\ 
   -B^T & \mathrm{Diag}(d^{(2)}) 
   \end{pmatrix} ,
\end{equation}
where $B$ is the entrywise square of $A$,
that is, $B_{ij} = A_{ij}^2$, and
for $i = 1,\ldots,n$
$$
    d^{(1)}_{i} = \sum_{j=1}^n A_{ij}^2 ,
    \qquad
    d^{(2)}_{i} = \sum_{j=1}^n A_{ji}^2 .
$$
The solution of \eqref{eq:betaF}
or \eqref{eq:normaleq}
generally has $O(n^3)$ computational cost.
We can reduce the problem of solving 
\eqref{eq:normaleq} to that of the solution of an $n\times n$ linear system.
First observe that, since the matrix $U^T W^2 U$ has rank $2n-1$ and $U^T W^2 U\uno =0$, then all the solutions of \eqref{eq:normaleq} are $[x; y]+\gamma \uno$, where $[x; y]$ is any solution and $\gamma$ is an arbitrary constant. However, the generators $x$ and $y$ are normalized exactly when the vector $[x;y]$ has zero sum. Therefore, we look for the solution $[x ; y] $ such that $\uno^T [x; y]=0$, so that we solve the system
\begin{equation}\label{eq:UTW2Usystem}
U^T W^2 U \begin{bmatrix}
    x \\ y
\end{bmatrix}= U^T W \uno.
\end{equation}
From the block structure \eqref{eq:UTW2U}, by applying one step of block Gaussian elimination, we find that \eqref{eq:UTW2Usystem} is equivalent to
\[
\begin{pmatrix}
  \operatorname{Diag}(d^{(1)}) & -B\\
  0 & S
\end{pmatrix}
\begin{bmatrix}
    x \\ y
\end{bmatrix} =
\begin{pmatrix}
  I & 0 \\
B^T  \operatorname{Diag}(d^{(1)})^{-1} & I
\end{pmatrix} \begin{bmatrix}
    b^{(1)} \\ -b^{(2)}
\end{bmatrix},
\]
where $S=\operatorname{Diag}(d^{(2)})-B^T \operatorname{Diag}(d^{(1)})^{-1} B$, that is equivalent to the following equations
\begin{align}
     x &= \operatorname{Diag}(d^{(1)})^{-1}( b^{(1)} + By),\label{eq:x}
    \\
   S y & = B^T  \operatorname{Diag}(d^{(1)})^{-1} b^{(1)} - b^{(2)}.\label{eq:y}
\end{align}
The matrix of the latter system is a singular M-matrix, such that $S\uno=0$. From the property $(S-\frac{1}{n}\uno \uno^T)^{-1}S=I-\frac{1}{n}\uno \uno^T$, we deduce that the general solution of \eqref{eq:y} is
\[
y=\left(S-\frac{1}{n}\uno \uno^T\right)^{-1} c+ \gamma \uno, 
\]
where $c= B^T  \operatorname{Diag}(d^{(1)})^{-1} b^{(1)} + b^{(2)}$ and $\gamma$ is an arbitrary constant.
Another approach to solve \eqref{eq:y} consists in considering the $(n-1)\times (n-1)$ nonsingular linear system $\hat S \hat y=\hat c$, where $\hat S$ is obtained by removing the last row and the last column of $S$ and $\hat c$ is obtained by removing the last entry of the right-hand side $c$. Then the vector $y=[\hat y; 0]$ is a particular solution of \eqref{eq:y} and all the solutions are $y+\gamma \uno$, with arbitrary $\gamma$. 

Once a solution $y$ of \eqref{eq:y} is computed, the vector $x$ can be recovered from \eqref{eq:x}. 
    Altogether, the solution of 
    \eqref{eq:betaF} can be computed at the  cost of
    $O(n^3)$ arithmetic operations. The resulting procedure is shown in Algorithm \ref{alg:last}.

\begin{algorithm}[ht]   \label{alg:last}
\LinesNumbered 
\SetNlSty{texttt}{}{.} 
\DontPrintSemicolon
\caption{Recovery of normalized generators from \eqref{eq:betaF} }   
\KwIn{Matrix $A = (A_{ij})$ with no zero entries}
\KwOut{Normalized generators $x,y$ }
\BlankLine
Compute the vectors $d^{(1)}$ and $d^{(2)}$ and the matrix $B$  in \eqref{eq:UTW2U} \;
$S=\operatorname{Diag}(d^{(2)})-B^T \operatorname{Diag}(d^{(1)})^{-1} B$ \;
Compute any solution $y$ in \eqref{eq:y} \;
Compute $x$ from \eqref{eq:x} \;
$\alpha = (\sum_{i=1}^n x_i + y_i)/(2n)$\;
$x = x - \alpha\uno$\;
$y = y - \alpha\uno$\;
\end{algorithm}


\section{Numerical experiments}
\label{sec:numerical}

For numerical tests, we implemented Algorithms \ref{alg:LL1}--\ref{alg:last} in MATLAB. In Algorithm~\ref{alg:last}, eq.~\eqref{eq:y} was solved by applying the backslash operator to the linear system $\left(S-\frac{1}{n}\uno \uno^T\right)y=c$.  
The experiments were run in MATLAB R2024b on a Dell XPS 13 9340 laptop equipped with an Intel Core Ultra 7 155H processor and $32$ GB of RAM, running Ubuntu 24.04.2 LTS. 

\begin{myexample}\label{ex:ex1}\rm

We choose a set of uniformly spaced and interlaced Cauchy points,
$x_i = i/n$ and $y_i = x_i + 1/(2n)$ for $i = 1,\ldots,n$,
and set $C = \cau(x,y)$.
This configuration emphasizes the performance differences among the algorithms.
Firstly, we set $n = 100$. 
For several noise levels $\delta\in[10^{-9},10^{-1}]$ we apply a multiplicative componentwise perturbation to $C$:
$$ A_{i,j}=(1\pm\delta)C_{i,j}$$
where the sign in the perturbation factor is chosen uniformly at random. 
The same sign pattern is used for all values of $\delta$.
This construction allows us to consider $\delta$ as a precise estimate of $\beta_{\max}(A)$, since the absolute value of every entry of the matrix $\nabla_{x,y}(A) - \uno\uno^T$ is $\delta$.

We then apply Algorithms \ref{alg:LL1}--\ref{alg:last} to $A$ and compute Cauchy matrices $C_1$, $C_2$, $C_3$, $C_4$ from the Cauchy points obtained from the algorithms in the previous sections. For Algorithm~\ref{alg:generaluv}, vectors $v$ and $w$ are chosen such that $v_j=w_j=2(n-j+1)/n(n+1)$, i.e., entries are positive, uniformly decreasing
and have unit sum.

Figure \ref{fig:esp1} shows the relative normwise errors 
$\|C-C_i\|_\fro/\|C\|_\fro$ (left panel) and $\|A-C_i\|_\fro/\|A\|_\fro$
(right panel). They visually appear to be very close to each other. Algorithm \ref{alg:last} provides the best approximation quality in this example by several orders of magnitude, while Algorithm 1 is the least accurate. The dotted line in the figure shows $\beta_{\max}(A)$, that is $\delta$.
This information is included in relation to Theorem \ref{thm:LL1}, since 
$\beta_{\max}(A)$ is a lower bound for the constant $\beta$ appearing in that theorem. Remarkably, Algorithm \ref{alg:last} produces errors that are below this value.

We also compare the performance of Algorithms \ref{alg:LL1}--\ref{alg:last} on matrices of increasing size. Here $\delta=10^{-5}$ is fixed, and matrix size increases from $100$ to $2000$. 
Figure \ref{fig:esp1_bis} shows relative normwise errors for the four algorithms. 
As in the previous case, Algorithm \ref{alg:last} yields the best approximation quality; moreover, in contrast to the other algorithms, the relative approximation errors are quite insensitive to $n$. 
\end{myexample}

\begin{figure}
    \centering
  \begin{tikzpicture}[scale=0.9]
  \usetikzlibrary{plotmarks}:
        \begin{loglogaxis}[
                legend style={at={(0.98,0.02)},anchor=south east},
                legend columns=2,
                width = .5\linewidth, height = .3\textheight,
                xlabel = {$\delta$}, 
                ylabel = {$\| C-C_i\|_\fro/\|C\|_\fro$}, 
                grid = major,
	major grid style = {lightgray},
	minor grid style = {lightgray!25}
                           ]
            \addplot 
            table[x index = 1, y index = 2] {esp1data.txt};
            \addplot 
            table[x index = 1, y index = 3] {esp1data.txt};
            \addplot [mark=triangle*, mark size=3pt, color=brown!70!black]
            table[x index = 1, y index = 4] {esp1data.txt};
            \addplot 
            table[x index = 1, y index = 5] {esp1data.txt};
            \addplot [color=green!40!black, style={dashed, line width=1pt}]
            table[x index = 1, y index = 1] {esp1data.txt};
          \legend{
          $i = 1$, 
          $i = 2$, 
          $i = 3$, 
          $i = 4$  
          }
        \end{loglogaxis}
  \end{tikzpicture} \hfill
  \begin{tikzpicture}[scale=0.9]
  \usetikzlibrary{plotmarks}:
        \begin{loglogaxis}[
                legend style={at={(0.98,0.02)},anchor=south east},
                legend columns=2,
                width = .5\linewidth, height = .3\textheight,
                xlabel = {$\delta$}, 
                ylabel = {$\| A-C_i\|_\fro/\|A\|_\fro$}, 
                grid = major,
	major grid style = {lightgray},
                           ],
            \addplot 
            table[x index = 1, y index = 6] {esp1data.txt};
            \addplot 
            table[x index = 1, y index = 7] {esp1data.txt};
            \addplot [mark=triangle*, mark size=3pt, color=brown!70!black]            
            table[x index = 1, y index = 8] {esp1data.txt};
            \addplot 
            table[x index = 1, y index = 9] {esp1data.txt};
            \addplot [color=green!40!black, style={dashed, line width=1pt}]
            table[x index = 1, y index = 1] {esp1data.txt};
          \legend{
          $i = 1$, 
          $i = 2$, 
          $i = 3$, 
          $i = 4$  
          }
        \end{loglogaxis}
        \end{tikzpicture}
    \caption{Relative normwise errors for Example \ref{ex:ex1}. Here $C$ is the original Cauchy matrix and $C_i$, $i=1,2,3,4$ are the Cauchy matrices recovered from the corresponding algorithms from $A$. The parameter $\delta$ measures the magnitude of entrywise relative perturbations applied to $C$.
    The dashed line represents $\beta_{\max}(A)$. 
    Matrix size is $100\times 100$.}
    \label{fig:esp1}
\end{figure}
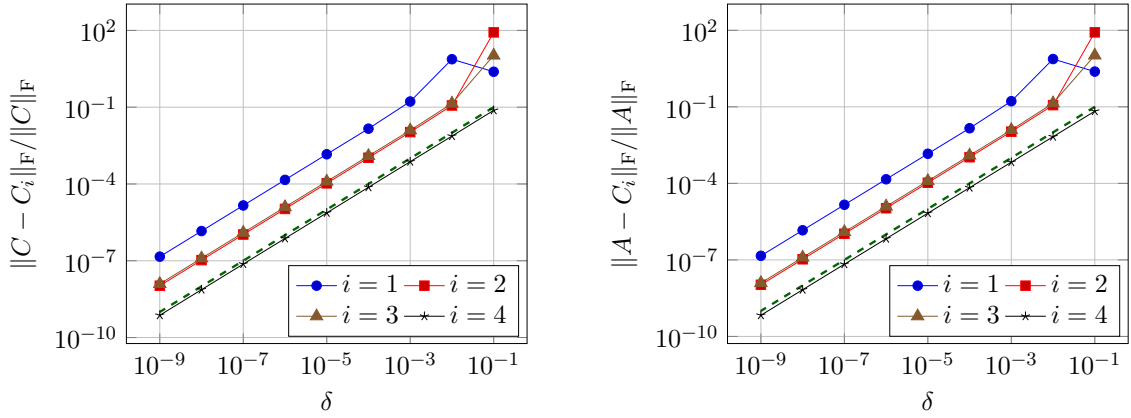

\begin{figure}
    \centering
      \begin{tikzpicture}[scale=0.9]
        \begin{semilogyaxis}[
               legend style={at={(0.98,0.15)},anchor=south east},
                legend columns=2,
                width = .5\linewidth, height = .3\textheight,
                xlabel = {$n$}, 
                ylabel = {$\| C-C_i\|_\fro/\|C\|_\fro$}, 
                xtick = {0,500,1000,1500,2000},  
                xticklabels = {$0$, $500$, $1000$, $1500$, $2000$},
                grid = major,
	major grid style = {lightgray},
	minor grid style = {lightgray!25}
                           ]
            \addplot table[x index = 1, y index = 2] {esp1bdata.txt};
            \addplot table[x index = 1, y index = 3] {esp1bdata.txt};
            \addplot [mark=triangle*, mark size=3pt, color=brown!70!black]
            table[x index = 1, y index = 4] {esp1bdata.txt};
            \addplot table[x index = 1, y index = 5] {esp1bdata.txt};
          \legend{
          $i = 1$, 
          $i = 2$, 
          $i = 3$, 
          $i = 4$  
                }
        \end{semilogyaxis}
    \end{tikzpicture} \hfill
  \begin{tikzpicture}[scale=0.9]
        \begin{semilogyaxis}[
               legend style={at={(0.98,0.15)},anchor=south east},
                legend columns=2,
                width = .5\linewidth, height = .3\textheight,
                xlabel = {$n$}, 
                ylabel = {$\| A-C_i\|_\fro/\|A\|_\fro$}, 
                xtick = {0,500,1000,1500,2000},  
                xticklabels = {$0$, $500$, $1000$, $1500$, $2000$},
                grid = major,
	minor tick num = 1,
	major grid style = {lightgray},
	minor grid style = {lightgray!25}
                           ]
            \addplot table[x index = 1, y index = 6] {esp1bdata.txt};
            \addplot table[x index = 1, y index = 7] {esp1bdata.txt};
            \addplot [mark=triangle*, mark size=3pt, color=brown!70!black]
            table[x index = 1, y index = 8] {esp1bdata.txt};
            \addplot table[x index = 1, y index = 9] {esp1bdata.txt};
          \legend{
          $i = 1$, 
          $i = 2$, 
          $i = 3$, 
          $i = 4$  
                }
        \end{semilogyaxis}
    \end{tikzpicture}
    \caption{Relative normwise errors for Example \ref{ex:ex1}. Here $C$ is the original Cauchy matrix and $C_i$, $i=1,2,3,4$ are the Cauchy matrices recovered from the corresponding algorithms from the perturbed matrix $A$ with $\delta=10^{-5}$. Matrix size increases from $100$ to $2000$ by steps of $100$.}
    \label{fig:esp1_bis}
\end{figure}
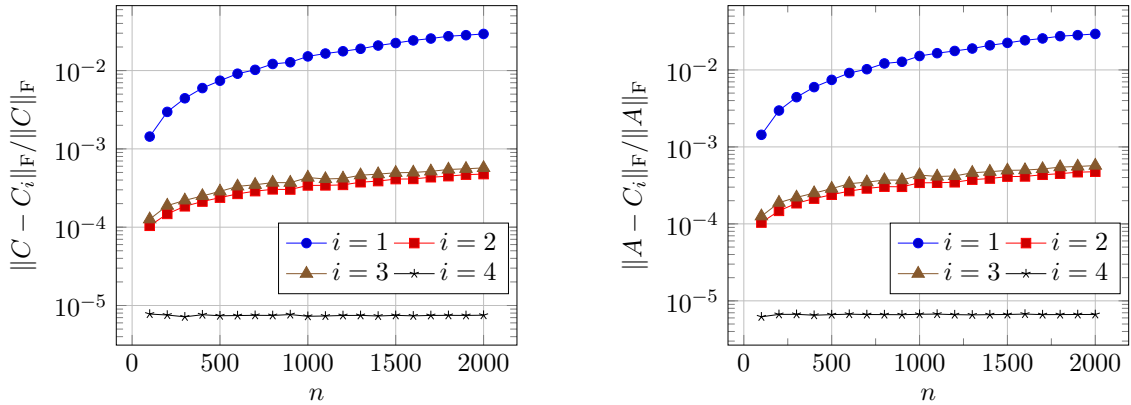

\begin{myexample}\label{ex:ex2}\rm

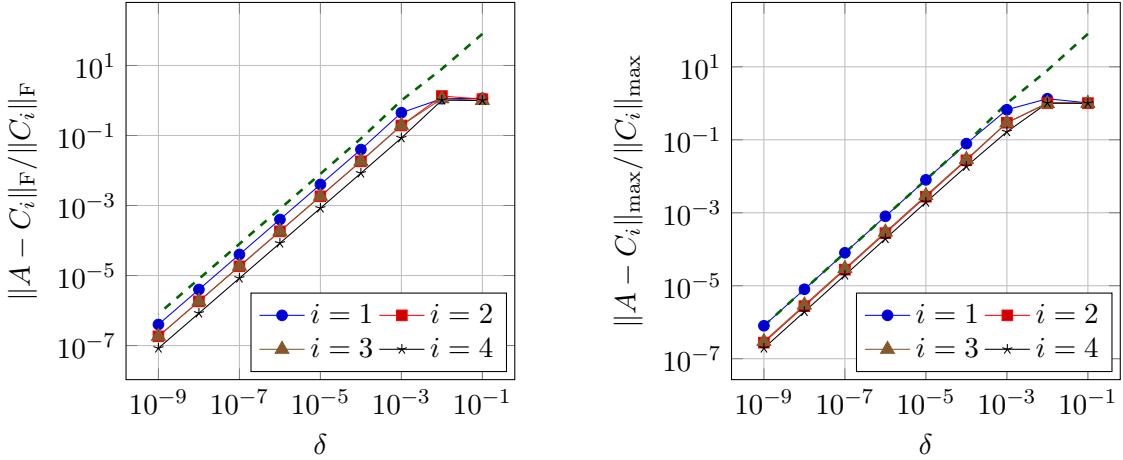
\begin{figure}
    \centering
      \begin{tikzpicture} 
        \begin{loglogaxis}[
               legend style={at={(0.98,0.02)},anchor=south east},
                legend columns=2,
                width = .45\linewidth, height = .3\textheight,
                xlabel = {$\delta$}, 
                ylabel = {$\| A-C_i\|_\fro/\|C_i\|_\fro$}, 
                grid = both,
	major grid style = {lightgray},
	minor grid style = {lightgray!25}
                           ]
            \addplot table[x index = 1, y index = 2] {esp2data.txt};
            \addplot table[x index = 1, y index = 3] {esp2data.txt};
            \addplot [mark=triangle*, mark size=3pt, color=brown!70!black]
            table[x index = 1, y index = 4] {esp2data.txt};
            \addplot table[x index = 1, y index = 5] {esp2data.txt};
            \addplot [color=green!40!black, style={dashed, line width=1pt}] 
                     table[x index = 1, y index = 10] {esp2data.txt};
          \legend{
          $i = 1$, 
          $i = 2$, 
          $i = 3$, 
          $i = 4$  
                }
        \end{loglogaxis}
    \end{tikzpicture} \hfill
  \begin{tikzpicture} 
        \begin{loglogaxis}[
               legend style={at={(0.98,0.02)},anchor=south east},
                legend columns=2,
                width = .45\linewidth, height = .3\textheight,
                xlabel = {$\delta$}, 
                ylabel = {$\| A-C_i\|_{\max}/\|C_i\|_{\max}$}, 
                grid = major,
	major grid style = {lightgray},
                           ]
            \addplot table[x index = 1, y index = 6] {esp2data.txt};
            \addplot table[x index = 1, y index = 7] {esp2data.txt};
            \addplot [mark=triangle*, mark size=3pt, color=brown!70!black]
            table[x index = 1, y index = 8] {esp2data.txt};
            \addplot table[x index = 1, y index = 9] {esp2data.txt};
            \addplot [color=green!40!black, style={dashed, line width=1pt}] 
                     table[x index = 1, y index = 10] {esp2data.txt};
          \legend{
          $i = 1$, 
          $i = 2$, 
          $i = 3$, 
          $i = 4$  
                }
        \end{loglogaxis}
    \end{tikzpicture}
    \caption{Relative normwise errors for Example \ref{ex:ex2}. Here $C_i$, $i=1,2,3,4$ are the Cauchy matrices recovered by the corresponding algorithms. The dashed lines represent
    the upper bound from \eqref{eq:thm3.5}. 
    Matrix size is $100\times 100$.}
    \label{fig:esp2}
\end{figure}

This experiment aims to illustrate the validity of a bound from Theorem \ref{thm:kappastargen}. 
We generate Cauchy points as in Example \ref{ex:ex1} with $n=100$. 
For several noise levels $\delta\in[10^{-9},10^{-1}]$ we apply an additive componentwise perturbation to $D = H(x,y)$, that is, we set $Z_{ij} = D_{ij} \pm \delta$,
where the sign in the perturbation is chosen uniformly at random. 
The same sign pattern is used for all values of $\delta$.
Then we set $A = \hinv{Z}$.
This construction allows us to consider $\delta$ as a precise estimate of $\mu_{\max}(A)$, since the absolute value of every entry of the matrix $\hinv{A} - D$ is $\delta$. 

For $i = 1,\ldots,4$, we denote $C_i$ the Cauchy matrices from the Cauchy points
computed by Algorithm $i$.
Figure \ref{fig:esp2} shows the approximation errors 
$\|A-C_i\|_{\fro}/\|C_i\|_{\fro}$ (left panel) and $\|A-C_i\|_{\max}/\|C_i\|_{\max}$ (right panel). The dotted line represents the quantity $4\mu_{\max}(A)\|A\|_{\max}$,
which is the approximation error bound in Theorem \ref{thm:kappastargen}. 
The `Cauchyness' criterion employed by 
Algorithm \ref{alg:last} appears to be more effective at recovering an approximation of the original matrix than the other algorithms.
By contrast, Algorithm \ref{alg:LL1} produces the worst approximation, with an error that nearly reaches the upper bound.
\end{myexample}

\begin{myexample}\label{ex:ex3}\rm

\begin{figure}
    \centering
\begin{tikzpicture}
\begin{axis}[
	xmin = 2.9, xmax = 3.9,
	ymin = -2.5, ymax = 1,
	width = .85\textwidth,
	height = 0.5\textwidth,
	xtick distance = 0.2,
	ytick distance = 0.5,
	grid = both,
	major grid style = {lightgray},
	minor grid style = {lightgray!25},
	xlabel = {$\log_{10}n$},
	ylabel = {$\log_{10}$(time)},
	legend cell align = {left},
	legend pos = north west
]

\addplot[
	teal, 
	only marks
] table[x index=0, y index=1] {time4datanew.txt};

\addplot[
	magenta, 
	only marks
] table[x index=0, y index=1] {ex4timing2.txt};

\addplot[
	very thick,
    dashed,
	orange
] table[
	x index=0,
	y = {create col/linear regression}
] {time4datanew.txt};
\xdef\slope4{\pgfplotstableregressiona}
\addplot[
	thick,
	blue
] table[
	x index=0,
	y = {create col/linear regression}
] {ex4timing2.txt};
\xdef\slopetwo{\pgfplotstableregressiona}
\addlegendentry{\small Timing data, Alg.4}
\addlegendentry{\small Timing data, Alg.2}
\addlegendentry{\small
	  $\alpha = {}$ \pgfmathprintnumber{\slope4}
    }
\addlegendentry{\small
	  $\alpha = {}$ \pgfmathprintnumber{\slopetwo}
    }
\end{axis}
\end{tikzpicture}
    
    \caption{Timings and linear fits for Algorithm
    \ref{alg:LL2} and \ref{alg:last}, in log-log scale. The slope of the linear fits gives the exponent in the power law $cn^\alpha$.}
    \label{fig:time4}
\end{figure}
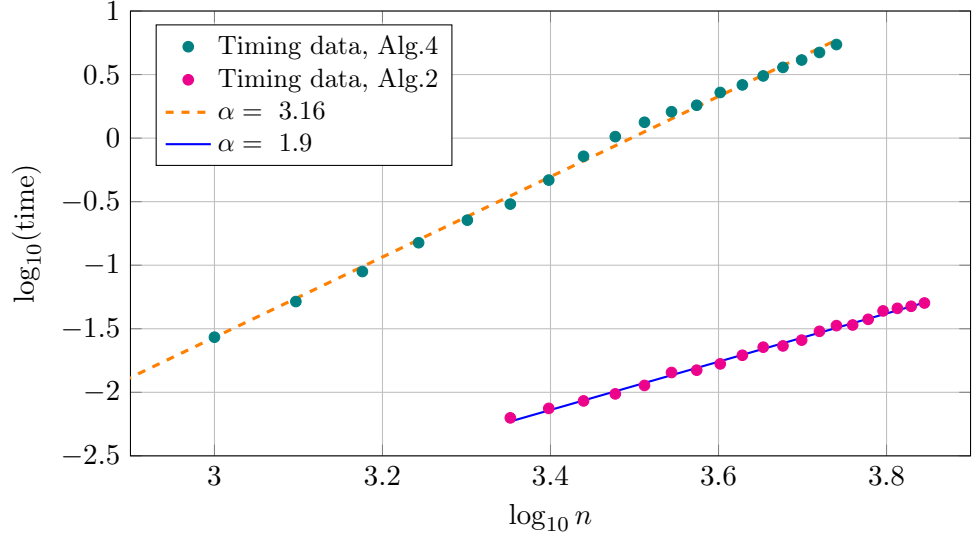  

As pointed out at the end of Section \ref{sec:displacementbased}, a complexity analysis for Algorithm \ref{alg:last} gives an asymptotic computational cost of $O(n^3)$ arithmetic operations. 
The analogous cost for Algorithm \ref{alg:LL2} is $O(n^2)$, due to the matrix-vector products in the algorithm.
In this example, we perform a timing test for Algorithms
\ref{alg:LL2} and \ref{alg:last} with $n$ ranging from $500$ to $5000$. 
For each matrix size, we compute the average time of ten runs.
Then we fit the observed execution times to power laws of the form $cn^\alpha$ using linear regressions on a log-log scale. The results are shown in Figure \ref{fig:time4}. 
The exponents computed by linear regressions closely match the theoretical values.
\end{myexample}

\begin{myexample}\label{ex:ex4}\rm

\begin{figure}
    \centering
      \begin{tikzpicture}[scale=0.9]
        \begin{semilogyaxis}[
               legend style={at={(0.98,0.5)},anchor=south east},
                legend columns=2,
                width = .5\linewidth, height = .3\textheight,
                xlabel = {$n$}, 
                ylabel = {$\| C-C_i\|_\fro/\|C\|_\fro$}, 
                xtick = {0,500,1000,1500,2000},  
                xticklabels = {$0$, $500$, $1000$, $1500$, $2000$},
                grid = both,
	major grid style = {lightgray},
	minor grid style = {lightgray!25}
                           ]
            \addplot table[x index = 1, y index = 2] {esp4data.txt};
            \addplot table[x index = 1, y index = 3] {esp4data.txt};
            \addplot [mark=triangle*, mark size=3pt, color=brown!70!black]
            table[x index = 1, y index = 4] {esp4data.txt};
            \addplot table[x index = 1, y index = 5] {esp4data.txt};
        \end{semilogyaxis}
    \end{tikzpicture} \hfill
  \begin{tikzpicture}[scale=0.9]
        \begin{semilogyaxis}[
               legend style={at={(0.98,0.02)},anchor=south east},
                legend columns=2,
                width = .5\linewidth, height = .3\textheight,
                xlabel = {$n$}, 
                ylabel = {$\| A-C_i\|_\fro/\|A\|_\fro$}, 
                xtick = {0,500,1000,1500,2000},  
                xticklabels = {$0$, $500$, $1000$, $1500$, $2000$},
                grid = major,
	major grid style = {lightgray},
                           ]
            \addplot table[x index = 1, y index = 6] {esp4data.txt};
            \addplot table[x index = 1, y index = 7] {esp4data.txt};
            \addplot [mark=triangle*, mark size=3pt, color=brown!70!black]
            table[x index = 1, y index = 8] {esp4data.txt};
            \addplot table[x index = 1, y index = 9] {esp4data.txt};
          \legend{
          $i = 1$, 
          $i = 2$, 
          $i = 3$, 
          $i = 4$  
                }
        \end{semilogyaxis}
    \end{tikzpicture}
    \caption{Relative normwise errors for Example \ref{ex:ex4}. Here $C_i$, $i=1,2,3,4$ are the Cauchy matrices recovered by the corresponding algorithms. The matrix $A$ is subject to an unbalanced perturbation with size $\delta=10^{-5}$. Matrix size increases from $100$ to $2000$ by steps of $100$.}
    \label{fig:esp4}
\end{figure}
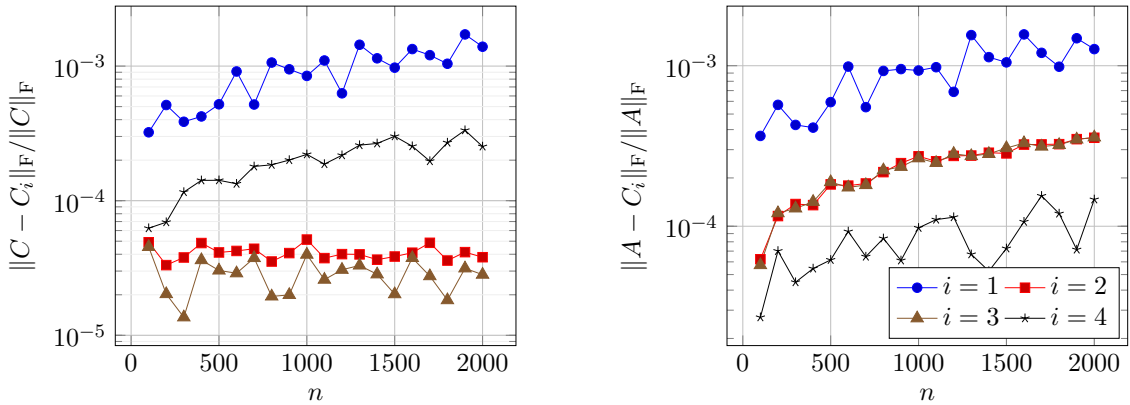

In Example \ref{ex:ex1}, the perturbation is evenly spread across the test matrix. It is therefore understandable that Algorithm \ref{alg:LL2} works slightly better than Algorithm \ref{alg:generaluv} applied with uniformly decreasing vectors $v$ and $w$. On a matrix where the perturbation mainly affects the trailing principal block, we expect the opposite behavior. To verify this conjecture numerically, we set up test data as in the second part of Example \ref{ex:ex1}, except that the perturbation applied to entry $(i,j)$ is weighted by a factor $n/(n-i+1)(n-j+1)$. Figure \ref{fig:esp4} shows that Algorithm \ref{alg:generaluv} does indeed provide 
the best reconstruction of the original Cauchy matrix. 
However, the best approximation to $A$ is given by Algorithm~\ref{alg:last}. 
\end{myexample}

\begin{myexample}\label{ex:ex5}\rm

\begin{figure}
    \centering
      \begin{tikzpicture}[scale=0.9]
        \begin{semilogyaxis}[
               legend style={at={(0.98,0.5)},anchor=south east},
                legend columns=2,
                width = .5\linewidth, height = .3\textheight,
                xlabel = {$n$}, 
                ylabel = {$\| Z-D_i\|_\fro$}, 
                xtick = {0,500,1000,1500,2000},  
                xticklabels = {$0$, $500$, $1000$, $1500$, $2000$},
                grid = both,
	major grid style = {lightgray},
	minor grid style = {lightgray!25}
                           ]
            \addplot table[x index = 1, y index = 2] {esp5data.txt};
            \addplot table[x index = 1, y index = 3] {esp5data.txt};
            \addplot [mark=triangle*, mark size=3pt, color=brown!70!black]
            table[x index = 1, y index = 4] {esp5data.txt};
            \addplot table[x index = 1, y index = 5] {esp5data.txt};
          \legend{
          $i = 1$, 
          $i = 2$, 
          $i = 3$, 
          $i = 4$  
                }
        \end{semilogyaxis}
    \end{tikzpicture} \hfill
  \begin{tikzpicture}[scale=0.9]
        \begin{semilogyaxis}[
               legend style={at={(0.98,0.5)},anchor=south east},
                legend columns=2,
                width = .5\linewidth, height = .3\textheight,
                xlabel = {$n$}, 
                ylabel = {$\| A-C_i\|_\fro/\|A\|_\fro$}, 
                xtick = {0,500,1000,1500,2000},  
                xticklabels = {$0$, $500$, $1000$, $1500$, $2000$},
                grid = major,
	major grid style = {lightgray},
                           ]
            \addplot table[x index = 1, y index = 6] {esp5data.txt};
            \addplot table[x index = 1, y index = 7] {esp5data.txt};
            \addplot [mark=triangle*, mark size=3pt, color=brown!70!black]
            table[x index = 1, y index = 8] {esp5data.txt};
            \addplot table[x index = 1, y index = 9] {esp5data.txt};
          \legend{
          $i = 1$, 
          $i = 2$, 
          $i = 3$, 
          $i = 4$  
                }
        \end{semilogyaxis}
    \end{tikzpicture}
    \caption{Normwise errors for Example \ref{ex:ex5}. 
    Here $C_i$, $i=1,2,3,4$ are the Cauchy matrices recovered from the corresponding algorithms, and $D_i = \hinv{C_i}$. The matrix $A = \hinv{Z}$ is subject to a worst-case perturbation with size $\delta=10^{-5}$. Matrix size increases from $100$ to $2000$ by steps of $100$.}
    \label{fig:esp5}
\end{figure}
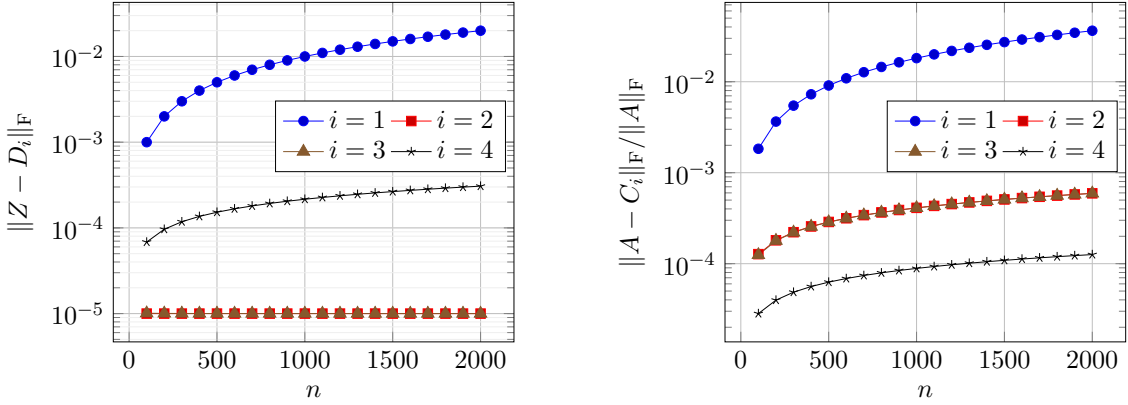

In this example, we focus on errors for componentwise inverses of Cauchy matrices. Recall that Theorem \ref{thm:kappastargen} provides an upper bound for such errors, which depends on the norm of the projection operator associated with the chosen algorithm. From the proof of Theorem \ref{thm:kappastargen}, one may also pinpoint worst-case perturbations that make the bound sharp (i.e., equality is attained). Such perturbations can be chosen as scalar multiples of the matrix obtained from a square columnwise reshape of the dominant right singular vector of the operator $M_w\otimes M_v$. For instance, for the operator $M_{e_1}\otimes M_{e_1}$ associated with Algorithm \ref{alg:LL1} such a worst-case perturbation is the $n\times n$ matrix 
$Y = vv^T/v^Tv$ where $v = (1,-1/(n-1),\ldots,-1/(n-1))^T\in\R^n$ is the right singular vector of $M_{e_1}$ associated with the dominant singular value, $\|M_{e_1}\|_2 = \sqrt{n}$. Hence, 
\begin{align*}
Y(1,1)&=\frac{n-1}{n},\\
Y(1,i)&=Y(i,1)=-\frac1n,\qquad & i=2,\ldots n,\\
Y(i,j)&= \frac{1}{n(n-1)},  & i,j=2,\ldots,n.
\end{align*}
Notably, this matrix is orthogonal to $\mathcal{D}$ with respect to the Frobenius inner product. This can be verified by computing the projection of $Y$ onto $\mathcal{D}$, using the projector $\Delta_2$ in Theorem \ref{thm:1+2}. The projection is zero, as $\uno^Tv = 0$.
Consequently, if $A$ is any matrix such that $\hinv{A} = D + \delta Y$ for some $\delta\geq 0$ and $D\in\mathcal{D}$ then $\mu_\fro(A) = \delta$.
Moreover, if $D_1\in\mathcal{D}$ denotes the matrix recovered by Algorithm \ref{alg:LL1}
then $\|\hinv{A} - D_1\|_\fro = \|M_{e_1}\|_2^2\, \mu_\fro(A) = n\delta$,
that is, we have equality in Theorem \ref{thm:1+2}. 

For $n$ ranging from $100$ to $2000$ we define a Cauchy matrix $C$ as in Example \ref{ex:ex1}, compute its componentwise inverse $D$ and consider the perturbed matrix
 $Z=D+\delta Y$,
with $\delta=10^{-5}$. Algorithms \ref{alg:LL1}--\ref{alg:last} are applied to $A=Z^{[-1]}$, and we recover matrices $D_i\in\mathcal{D}$ and $C_i = \hinv{D_i}$ for $i=1, 2, 3, 4$. 
The left panel in Figure \ref{fig:esp5} shows errors $\|Z-D_i\|_\fro$.
In accordance with Theorem \ref{thm:kappastargen} we expect 
$\|Z-D_1\|_\fro = n\delta$ and $\|Z-D_2\|_\fro = \delta$,
as it actually occurs numerically. The relative errors $\|A-C_i\|_\fro/\|A\|_\fro$ shown in the right panel increase as $O(n)$,
corresponding to the growth of $\|A\|_{\max}$. Also in this example, 
the results of Algorithms \ref{alg:LL2} and \ref{alg:generaluv}
are almost indistinguishable, while Algorithm~\ref{alg:last} provides the best approximation to $A$, but not to $C$.
\end{myexample}


\begin{myexample}\label{ex:ex6}\rm

The approach adopted in the present work, as well as in \cite{liesen2016fast}, does not allow for perturbations of a Cauchy matrix $C$ that zero out one or more entries. It is interesting to explore the behavior of Algorithms \ref{alg:LL1}--\ref{alg:last} when $C$ is perturbed so that some entries become close to zero.

{\bf 1.} For $n=100$ define Cauchy points and the associated Cauchy matrix $C$ as in Example \ref{ex:ex1}. The perturbed matrix $A$ is also defined as in Example~\ref{ex:ex1}, with $\delta=10^{-4}$, except for entry $A_{1,100}$, which is set to $-10^{-5}$. As a consequence, the relative perturbation on entry $(1,100)$ is very close to $1-10^{-5}$. Indices $(1,100)$ have been chosen because $C_{1,100}$ is the entry of $C$ with the smallest absolute value. We run Algorithms \ref{alg:LL1}--\ref{alg:last} on $A$ and obtain the relative errors in Table \ref{tab:ex61}. We find $\mu_{\fro}=9.8999e+4$ and $\beta_{\fro}=1.0000$.

\begin{table}
\centering
\begin{tabular}{c|cc}
\toprule
& $\|C-C_i\|_F/\|C\|_F$ & $\|A-C_i\|_F/\|A\|_F$\\
\midrule
Algorithm 1 & $0.0726$ & $0.0727$ \\
Algorithm 2 & $1.0000$ & $1.0000$ \\
Algorithm 3 & $1.4024$ & $1.4024$ \\
Algorithm 4 & $7.5695e-5$ & $3.2929e-4$\\
\bottomrule
\end{tabular}
\caption{Results for Example \ref{ex:ex6}. Here $C_i$, $i=1,2,3,4$ are the Cauchy matrices recovered from the corresponding algorithms.}\label{tab:ex61}
\end{table}

{\bf 2.} Set $n=100$, $\delta=10^{-4}$, $\eta=10^{-8}$. Define Cauchy points and the associated Cauchy matrix $C$ as in Example \ref{ex:ex1}, except for $x_1$, which is set to $x_1=y_{100}+1/(\delta+\eta)$. It holds $\|C\|_F=3.1095e+3$.
Let $E$ be an $n\times n$ matrix with entries equal to $\pm 1$, where the sign is chosen uniformly at random, except for $E_{1,100}=-1$.
The perturbed matrix $A$ is defined as $A=C+\delta E$ and it holds $A_{1,100}=\eta$.
 We run Algorithms \ref{alg:LL1}-\ref{alg:last} on $A$ and obtain the absolute errors in Table \ref{tab:ex62}. We find $\mu_F=1.1909e+10$ and $\beta_F=7.7458$.

\begin{table}
\centering
\begin{tabular}{c|cc}
\toprule
& $\|C-C_i\|_F$ & $\|A-C_i\|_F$\\
\midrule
Algorithm 1 & $3.1016e+3$ & $3.1016e+3$ \\
Algorithm 2 & $3.1095e+3$ & $3.1095e+3$ \\
Algorithm 3 & $3.1095e+3$ & $3.1095e+3$ \\
Algorithm 4 & $0.0086$ & $0.0131$ \\
\bottomrule
\end{tabular}
\caption{Results for Example \ref{ex:ex6}. Here $C_i$, $i=1,2,3,4$ are the Cauchy matrices recovered from the corresponding algorithms.}\label{tab:ex62}
\end{table}

{\bf 3.} For $n=100$ define Cauchy points and the associated Cauchy matrix $C$ as in Example \ref{ex:ex1}. It holds 
$\|C\|_F=3.1174e+3$. Now consider the perturbed matrix $A$, where $A_{i,j}=10^{-12}$ for three  pairs of indices $(i_1,j_1)=(1,10)$, $(i_2,j_2)=(15,65)$, $(i_3,j_3)=(71,22)$, and $A_{i,j}=0$ otherwise. Let $E=A-C$ be the perturbation matrix; its norm is $\|E\|_F=10.5263$. We run Algorithms \ref{alg:LL1}--\ref{alg:last} on $A$ and obtain the absolute errors in Table \ref{tab:ex63}. Clearly, Algorithm \ref{alg:last} benefits from avoiding explicit pointwise inversion. Also note that the distance $\|A-C_4\|_F$ is only slightly larger than $\|E\|_F$. We have $\mu_F=1.7148e+12$ and $\beta_F=1.7321$.

\begin{table}
\centering
\begin{tabular}{c|cc}
\toprule
& $\|C-C_i\|_F$ & $\|A-C_i\|_F$\\
\midrule
Algorithm 1 & $312.3873$ & $312.2229$ \\
Algorithm 2 & $3.1174e+3$ & $3.1174e+3$ \\
Algorithm 3 & $3.1174e+3$ & $3.1174e+3$ \\
Algorithm 4 & $1.0721e-10$ & $10.9076$ \\
\bottomrule
\end{tabular}
\caption{Results for Example \ref{ex:ex6}. Here $C_i$, $i=1,2,3,4$ are the Cauchy matrices recovered from the corresponding algorithms.}\label{tab:ex63}
\end{table}
\end{myexample}


\begin{myexample}\label{ex:ex8}\rm

We examine here a simple electrostatics inverse problem. Consider a one-dimensional experimental setup where $n$ unit charges are placed at positions $x_1, \ldots, x_n$ and $n$ sensors are placed at $y_1,\dots,y_n$, with $y_j>x_i$ for all indices $i,j$. According to the laws of electrostatics, the potential at point $y_j$ induced by the $i$-th charge is, up to a constant, $V_{ij}=\frac{1}{x_i-y_j}$. By activating one charge at a time and reading out the measurements of the sensors, one forms the $n\times n$ 
Cauchy matrix $V=(V_{ij})$. In practice, due to measurement errors, the Cauchy structure is perturbed, and the experiment yields an unstructured matrix $A$. We may also assume that the position of the charges and of the sensors are not known exactly, except, say, $y_1$ for normalization purposes. Recovery of the Cauchy structure provides a physically correct model and allows us to retrieve the positions of charges and sensors.

For the purposes of this numerical example we choose points $x_i$ equispaced in the interval $[-1,-10^{-5}]$ and points $y_j$ equispaced in $[10^{-5},1]$. Multiplicative noise is applied to the Cauchy matrix $V$ as in Example \ref{ex:ex1}. Figure \ref{fig:espGreen1} shows relative normwise errors on reconstructed Cauchy matrices, whereas Figure \ref{fig:espGreen2} shows relative normwise errors on reconstructed Cauchy points, i.e., positions of charges and sensors.

\begin{figure}
    \centering
  \begin{tikzpicture}[scale=0.9]
  \usetikzlibrary{plotmarks}:
        \begin{loglogaxis}[
                legend style={at={(0.98,0.02)},anchor=south east},
                legend columns=2,
                width = .5\linewidth, height = .3\textheight,
                xlabel = {$\delta$}, 
                ylabel = {$\| V-C_i\|_\fro/\|V\|_\fro$}, 
                grid = major,
	major grid style = {lightgray},
	minor grid style = {lightgray!25}
                           ]
            \addplot 
            table[x index = 1, y index = 2] {esp1Greendatamat.txt};
            \addplot 
            table[x index = 1, y index = 3] {esp1Greendatamat.txt};
            \addplot [mark=triangle*, mark size=3pt, color=brown!70!black]
            table[x index = 1, y index = 4] {esp1Greendatamat.txt};
            \addplot 
            table[x index = 1, y index = 5] {esp1Greendatamat.txt};
          \legend{
          $i = 1$, 
          $i = 2$, 
          $i = 3$, 
          $i = 4$  
          }
        \end{loglogaxis}
  \end{tikzpicture} \hfill
  \begin{tikzpicture}[scale=0.9]
  \usetikzlibrary{plotmarks}:
        \begin{loglogaxis}[
                legend style={at={(0.98,0.02)},anchor=south east},
                legend columns=2,
                width = .5\linewidth, height = .3\textheight,
                xlabel = {$\delta$}, 
                ylabel = {$\| A-C_i\|_\fro/\|A\|_\fro$}, 
                grid = major,
	major grid style = {lightgray},
                           ],
            \addplot 
            table[x index = 1, y index = 6] {esp1Greendatamat.txt};
            \addplot 
            table[x index = 1, y index = 7] {esp1Greendatamat.txt};
            \addplot [mark=triangle*, mark size=3pt, color=brown!70!black] 
            table[x index = 1, y index = 8] {esp1Greendatamat.txt};
            \addplot  
            table[x index = 1, y index = 9] {esp1Greendatamat.txt};
          \legend{
          $i = 1$, 
          $i = 2$, 
          $i = 3$, 
          $i = 4$  
          }
        \end{loglogaxis}
        \end{tikzpicture}
    \caption{Relative matrix normwise errors for Example \ref{ex:ex8}. Here $V$ is the original Cauchy matrix and $C_i$, $i=1,2,3,4$ are the Cauchy matrices recovered from the corresponding algorithms from $A$. The parameter $\delta$ measures the magnitude of entrywise relative perturbations applied to $V$.
    Matrix size is $100\times 100$.}
    \label{fig:espGreen1}
\end{figure}

\begin{figure}
    \centering
  \begin{tikzpicture}[scale=0.9]
  \usetikzlibrary{plotmarks}:
        \begin{loglogaxis}[
                legend style={at={(0.98,0.02)},anchor=south east},
                legend columns=2,
                width = .5\linewidth, height = .3\textheight,
                xlabel = {$\delta$}, 
                ylabel = {$\| x-x^{(i)}\|_2/\|x\|_2$}, 
                grid = major,
	major grid style = {lightgray},
	minor grid style = {lightgray!25}
                           ]
            \addplot 
            table[x index = 1, y index = 2] {esp1Greendatavec.txt};
            \addplot 
            table[x index = 1, y index = 3] {esp1Greendatavec.txt};
            \addplot [mark=triangle*, mark size=3pt, color=brown!70!black]
            table[x index = 1, y index = 4] {esp1Greendatavec.txt};
            \addplot 
            table[x index = 1, y index = 5] {esp1Greendatavec.txt};
          \legend{
          $i = 1$, 
          $i = 2$, 
          $i = 3$, 
          $i = 4$  
          }
        \end{loglogaxis}
  \end{tikzpicture} \hfill
  \begin{tikzpicture}[scale=0.9]
  \usetikzlibrary{plotmarks}:
        \begin{loglogaxis}[
                legend style={at={(0.98,0.02)},anchor=south east},
                legend columns=2,
                width = .5\linewidth, height = .3\textheight,
                xlabel = {$\delta$}, 
                ylabel = {$\| y-y^{(i)}\|_2/\|y\|_2$},    
                grid = major,
	major grid style = {lightgray},
                           ],
            \addplot 
            table[x index = 1, y index = 6] {esp1Greendatavec.txt};
            \addplot 
            table[x index = 1, y index = 7] {esp1Greendatavec.txt};
            \addplot [mark=triangle*, mark size=3pt, color=brown!70!black] 
            table[x index = 1, y index = 8] {esp1Greendatavec.txt};
            \addplot 
            table[x index = 1, y index = 9] {esp1Greendatavec.txt};
          \legend{
          $i = 1$, 
          $i = 2$, 
          $i = 3$, 
          $i = 4$  
          }
        \end{loglogaxis}
        \end{tikzpicture}
    \caption{Relative vector normwise errors for Example \ref{ex:ex8}. Here $x$, $y$ are the vectors of original Cauchy points and $x^{(i)}$, $y^{(i)}$, $i=1,2,3,4$ are the vectors of  Cauchy points recovered from the corresponding algorithms applied to $A$. The parameter $\delta$ measures the magnitude of entrywise relative perturbations applied to $V$.
    Problem size is $n=100$.}
    \label{fig:espGreen2}
\end{figure}
\end{myexample}


\section*{Acknowledgements}
This work was partially supported by the Italian Ministry of University and Research (MUR) through the PRIN 2022 ``Low-rank Structures and Numerical Methods in Matrix and Tensor Computations and their Application'' code 20227PCCKZ MUR D.D. financing decree n.\ 104 of February 2nd, 2022 (CUP I53D23002280006 and CUP E53D23005520006), and through the MUR Excellence Department Project awarded to the Department of Mathematics, University of Pisa, CUP I57G22000700001.
The authors are also affiliated
to the INdAM-GNCS (Gruppo Nazionale di Calcolo Scientifico). 

\appendix
\section{A different parametrization of $\D$}\label{sec:xyalpha}

Let $\Rn0$ denote the set of zero-sum vectors in $\R^n$:
$$
   \Rn0 = \{x\in\R^n: \uno^T x  = 0\} .
$$
The matrix space $\mathcal{D}$
introduced in \eqref{eq:defD}
admits the following alternative description:
$$
   \mathcal{D} = \{\hat x\uno^T + \uno \hat y^T + \alpha \uno\uno^T:
   \hat x,\hat y\in \Rn0,\ \alpha\in\R\} .
$$
More precisely, every matrix $D\in\mathcal{D}$ can be written 
in a unique way as 
\begin{equation}\label{eq:ourdecomp}
   D = \hat x\uno^T + \uno \hat y^T + \alpha \uno\uno^T
\end{equation}
for some  $\hat x,\hat y\in \Rn0$ and $\alpha\in\R$. It is interesting to study the properties of this decomposition, which differs from the one using normalized generators introduced in \cite{liesen2016fast} and recalled in Section \ref{sec:preliminaries}.
Notably, the three terms in the right-hand side of \eqref{eq:ourdecomp} are mutually orthogonal
for the Frobenius inner product.
Henceforth, for any $D\in\D$ let $\mathcal{R}(D)$ be the vector
$[\hat x; \hat y; \sqrt{n}\alpha]\in\R^{2n+1}$ identified by \eqref{eq:ourdecomp}.
We immediately obtain that 
$\|D\|_\fro = \sqrt{n}\|\mathcal{R}(D)\|_2$.
Indeed, the orthogonality of the three terms in \eqref{eq:ourdecomp} gives
$$
   \|D\|_\fro^2 = \|\hat x\uno^T\|_\fro^2 + \|\uno \hat y^T\|_\fro^2 + 
   \|\alpha \uno\uno^T\|_\fro^2
   = n \big( \|\hat x\|_2^2 + \|\hat y\|_2^2 +n\alpha^2 \big)
   = n \|\mathcal{R}(D)\|_2^2 .
$$
Thus, apart of the factor $\sqrt{n}$, the map 
$\mathcal{R}$ induces an isometry between the metric space $(\D,\|\cdot\|_\fro)$ and the subspace $\Rn0\times\Rn0\times\R$ of $\R^{2n+1}$ endowed with the $2$-norm.
As such, it produces a perfectly well conditioned representation of matrices in $\D$.
This motivates us to evaluate the conditioning of the representation 
provided by normalized generators, as shown below.

\begin{theorem}
Let $\bar x$ and $\bar y$ be the normalized generators 
of $D\in\D$. Then
$$
   \|\mathcal{R}(D)\|_2 \leq \| [\bar x; \bar y] \|_2
   \leq \sqrt{2} \|\mathcal{R}(D)\|_2 ,
$$
with attainable equalities.
\end{theorem}

\begin{proof}
Let $x,y$ be arbitrary generators, not necessarily normalized,
of a given matrix $D\in\D$,
and let $\sigma_x = \uno^T x/n$, $\sigma_y = \uno^T y/n$.
As recalled in Section \ref{sec:preliminaries}, $x,y$ are normalized if and only if $\sigma_x = - \sigma_y$.
Now,
$$
    \uno^TD\uno = \uno^T(x\uno^T - \uno y^T)\uno = n^2(\sigma_x - \sigma_y) .
$$
Furthermore, using \eqref{eq:ourdecomp},
$$
    \uno^TD\uno = \uno^T(\hat x\uno^T + \uno \hat y^T + \alpha \uno\uno^T)\uno = 
    n^2\alpha .
$$
Thus $\alpha = \sigma_x - \sigma_y$. 
Hence, given $x$ and $y$, the representation $\mathcal{R}(D)$ can be obtained from the linear relationships
\begin{align*}
    \hat x & = x - \sigma_x \uno \\
    \hat y & = y - \sigma_y \uno \\ 
    \sqrt{n}\alpha & = \sqrt{n}(\sigma_x - \sigma_y) . 
\end{align*}
These identities can be rewritten in matrix-vector form as $\mathcal{R}(D) = V [x;y]$, where
$$
   V = \begin{pmatrix} I_n - \uno\uno^T/n & O \\
   O & I_n - \uno\uno^T/n \\ 
   \uno_n^T/\sqrt{n} & -\uno_n^T/\sqrt{n} \end{pmatrix} \in\R^{(2n+1)\times(2n)} .
$$
Thus $\|\mathcal{R}(D)\|_2 \leq \|V\|_2\| [x; y] \|_2$,
which is notably true for $[x; y] = [\bar x;\bar y]$. 
Conversely, for any given $\mathcal{R}(D)$, 
the least $2$-norm solution of the linear equation
$\mathcal{R}(D) = V [x;y]$ gives the normalized generators of $D$,
that is, $[\bar x; \bar y] = V^+\mathcal{R}(D)$.
This implies $\|[\bar x; \bar y]\|_2 \leq \|V^+\|_2\|\mathcal{R}(D)\|_2$.

Let $Q_n\in\R^{n\times(n-1)}$ be a matrix whose columns form an orthonormal basis of the subspace $\Rn0$. Straightforward computations reveal that an SVD of $V$ is the following:
$$
   V = \begin{pmatrix} \textstyle
       0 & Q_n & O & \frac{1}{\sqrt{n}}\uno_n & 0 \\ 
       0 & O & Q_n & 0 & \frac{1}{\sqrt{n}}\uno_n \\ 
       1 & 0 & 0 & 0 & 0 
   \end{pmatrix}
   \Sigma
   \begin{pmatrix} \textstyle
       \frac{1}{\sqrt{2n}}\uno_n & Q_n & O & \frac{1}{\sqrt{2n}}\uno_n \\ 
       -\frac{1}{\sqrt{2n}}\uno_n & O & Q_n & \frac{1}{\sqrt{2n}}\uno_n 
   \end{pmatrix}^T ,
$$
with $\Sigma = \mathrm{Diag}(\sqrt{2}, 1, \ldots,1,0)\in\R^{(2n+1)\times(2n)}$. 
In particular, $\|V\|_2 = \sqrt{2}$
and $\|V^+\|_2 = 1$. This gives the inequalities in the claim.
The leftmost inequality holds as an identity when
both $\bar x$ and $\bar y$ are in $\Rn0$, while the 
equality on the right is reached when, e.g., $\bar x = \uno_n$ and $\bar y = -\uno_n$.
\end{proof}

The following result is analogous to Theorem 3.6 in \cite{liesen2016fast} (restricted to square matrices) but
pertains to the representation \eqref{eq:ourdecomp}
instead of the one with normalized generators. 
Our representation is shown to have a tighter error bound than the other one when recovering a perturbed Cauchy matrix via Algorithm \ref{alg:LL2}.

\begin{theorem}   \label{thm:almost3.6}
Let $A = \hinv{D}+N$ be a perturbed Cauchy matrix with nonzero entries, where $D\in\D$.
Let $0\leq \gamma < 1$ be a constant such that
$|D_{ij}N_{ij}| \leq \gamma$
for all $i,j=1,\ldots,n$. 
Moreover, let $\widetilde D\in\D$ be the matrix obtained from the output of Algorithm \ref{alg:LL2} applied to $A$.
Then
$$
   \frac{\|\mathcal{R}(\widetilde D) - \mathcal{R}(D)\|_2}
   {\|\mathcal{R}(D)\|_2} \leq \frac{\gamma}{1-\gamma} .
$$
\end{theorem}

\begin{proof}
Let $E = D - \hinv{A}$. By hypothesis,
$$
   E_{ij} = D_{ij} - \frac{1}{A_{ij}}
   = D_{ij} - \frac{D_{ij}}{1+D_{ij}N_{ij}} 
   = D_{ij} \frac{D_{ij}N_{ij}}{1-D_{ij}N_{ij}} .
$$
Therefore $|E_{ij}| \leq \gamma|D_{ij}|/(1-\gamma)$ and,
consequently, $\|E\|_\fro \leq \gamma\|D\|_\fro/(1-\gamma)$.
In the notation of Theorem \ref{thm:1+2}, we have
$$
   \widetilde D = \Delta_2(\hinv{A}) = 
   \Delta_2(D) + \Delta_2(E) = D + \Delta_2(E) .
$$
Furthermore,
$$
   \|\widetilde D - D\|_\fro = \|\Delta_2(E)\|_\fro 
   \leq \|E\|_\fro \leq \|D\|_\fro\frac{\gamma}{1-\gamma} ,
$$
since $\Delta_2$ is an orthogonal projector.
It remains to note that
$$
   \frac{\|\mathcal{R}(\widetilde D) - \mathcal{R}(D)\|_2}
   {\|\mathcal{R}(D)\|_2} =
   \frac{\|\mathcal{R}(\widetilde D - D)\|_2}
   {\|\mathcal{R}(D)\|_2} 
   = \frac{\|\widetilde D - D\|_2}{\|D\|_2} 
   \leq \frac{\gamma}{1-\gamma} ,
$$
and the proof is complete.
\end{proof}


\bibliographystyle{plain}

\end{document}